\documentclass[reqno,oneside,10pt]{amsart}

\usepackage{hyperref, color}
\usepackage{graphicx}
\usepackage{enumerate}
\usepackage[all]{xy}
\usepackage{amsmath,amsfonts,amssymb}
\usepackage[latin9]{inputenc}

\newtheorem{prop}{Proposition}[section]

\newtheorem{thm}[prop]{Theorem}

\theoremstyle{definition}
\newtheorem{defi}[prop]{Definition}
\newtheorem{exmp}[prop]{Example}

\theoremstyle{remark}

\newtheorem{remarks}[prop]{Remarks}

\newcommand{\benu}{\begin{enumerate}}
\newcommand{\enu}{\end{enumerate}}

\newcommand{\beqna}{\begin{eqnarray}}
\newcommand{\eqna}{\end{eqnarray}}
\newcommand{\beqnast}{\begin{eqnarray*}}
\newcommand{\eqnast}{\end{eqnarray*}}
\newcommand{\beqn}{\begin{equation}}
\newcommand{\eqn}{\end{equation}}
\newcommand{\beqnst}{\begin{equation*}}
\newcommand{\eqnst}{\end{equation*}}

\usepackage{latexsym}

\usepackage{xspace}
\usepackage{amscd}
\usepackage{amssymb}
\usepackage{amsfonts}

\makeatother

\newcommand{\bema}{\left ( \begin{array}}
\newcommand{\ema}{\end{array} \right )}

\newcommand{\End}{\operatorname{End}} 

\newcommand{\anti}{S^{-1}}

\begin{document}

\title[Partial actions]{Partial actions: what they are and why we care}

\author[E. Batista]{Eliezer Batista}
\address{Departamento de Matem\'atica, Universidade Federal de Santa Catarina, Brazil}
\email{eliezer1968@gmail.com}



\begin{abstract}
We present a survey of recent developments in the theory of partial actions of groups and Hopf algebras. 
\end{abstract} 

\maketitle

\setcounter{tocdepth}{2}

\flushbottom

\section{Introduction}
The history of mathematics is composed of many moments in which an apparent dead end just opens new doors and triggers further progress. The origins of partial actions of groups can be considered as one such case. The notion of a partial group action appeared for the first time in Exel's paper \cite{ruy}. The problem was to calculate the $K$-theory of some $C^*$-algebras which have an action by automorphisms of the circle $\mathbb{S}^1$. By a known result due to Paschke, under suitable conditions about a $C^*$-algebra carrying an action of the circle group, it can be proved that this $C^*$-algebra is isomorphic to the crossed product of its fixed point subalgebra by an action of the integers. The main hypothesis of Paschke's theorem is that the action has large spectral subspaces. A typical example in which this hypothesis does not occur is the action of $\mathbb{S}^1$ on the Toeplitz algebra by conjugation by the diagonal unitaries $\mbox{diag}(1,z,z^2, z^3, \ldots )$, for $z\in \mathbb{S}^1$. Therefore a new approach was needed to explore the internal structure of those algebras that carry circle actions but don't have large enough spectral subspaces. Using some techniques coming from dynamical systems, these algebras could be characterized as crossed products by partial automorphisms. By a partial automorphism of an algebra $A$ one means an isomorphism between two ideals in the algebra $A$. A partial action of a group $G$ on an algebra $A$ is a family of partial automorphisms of an algebra $A$ indexed by the elements of the group whose composition, when it is defined, must be compatible with the operation of $G$. The point of view of crossed products by partial actions of groups was enormously successful for classifying  $C^*$-algebras, the most relevant examples are the Bunce-Deddens algebras \cite{E-2}, approximately finite dimensional algebras \cite{E-3} and the Cuntz-Krieger algebras \cite{EL}. More recently, a characterization as a partial crossed product was given to the Bost-Connes and Cuntz-Li 
$C^*$-algebras associated to integral domains \cite{BE}. 

Very early, in the beginnings of the theory of partial actions of groups in the context of operator algebras, many unexpected connetions were explored. For example, partial group actions are closely related to actions of inverse semigroups \cite{ruy2}. In fact, for each group $G$, there is a universal inverse semigroup $S(G)$, nowadays known as Exel's semigroup, which associates to each partial action of $G$ on a set (topological space) $X$, a morphism of semigroups between $S(G)$ and the inverse semigroup of partially defined bijections (homeomorphisms) in $X$. At the level of $C^*$-algebras, this can be encoded by the notion of a partial representation of the group $G$, and the inverse semigroup $S(G)$ can be replaced by a $C^*$-algebra, $C^*_{par} (G)$ which has the universal property of factorizing partial representations by $\ast$-homomorphisms. The structure and the universal property of the partial crossed product was also explored \cite{quigg}. The study of partial dynamical systems, that is, dynamical systems originating from the action of a partially defined homeomorphism on a topological space, benefited strongly from the theory of partial group actions and partial crossed products; the most influential article in this direction is due to Exel, Laca and Quigg \cite{ELQ}. The historical developments and more recent advances of the theory of partial actions in operator algebras and dynamical systems are explained in Exel's book \cite{E-4}.

Partial group actions began to draw the attention of algebraists after the release of the seminal paper \cite{dok}, by R. Exel and M. Dokuchaev. There, the authors defined partial group actions on algebras and partial skew group algebras. The first difference between the pure algebraic and the operator algebra context appeared in the definition of the partial skew group algebra, which is not automatically associative in the general situation. A sufficient condition to assure the associativity is to require that the algebra $A$, on which the group acts partially, is semiprime, that is, all its non zero ideals are either idempotent or nondegenerate. This condition is satisfied by all $C^*$-algebras, because one can find an approximate unit in each of its ideals. The globalization problem, which is the question whether given a partial group action is a restriction of a global one, first addressed in the operator algebraic context by Abadie \cite{Abadie, AbadieTwo},  was also translated into purely algebraic terms in \cite{dok}. A partial action of a group $G$ on a unital algebra $A$ is globalizable if and only if, for each $g\in G$ the ideal $A_g$, which is the domain of the partially defined isomorphism $\alpha_{g^{-1}}$, is a unital ideal, that is, it is generated by a central idempotent $1_g \in A$. Finally, in the same article, the notion of a partial representation of a group $G$ and the construction of the partial group algebra $k_{par}G$, which factorizes partial representations by morphisms of algebras, was made in purely algebraic terms. The partial group algebra can be characterized as the partial skew group algebra defined by a partial action of the group $G$ on a specific abelian subalgebra of $k_{par}G$.

The algebraic theory of partial actions and partial representations of groups underwent several advances. The interconnection between partial actions and the theory of inverse semigroups were explored, for example in references \cite{BCFP}, \cite{EV} and \cite{St1}. In particular, Exel's universal semigroup $S(G)$ was shown in \cite{KL} to be isomorphic to the Birget-Rhodes expansion $\tilde{G}^R$ of the group $G$ \cite{BR1,BR2,Sz}. The role of groupoids in partial actions was first pointed out by F. Abadie in \cite{Aba2}. In fact, there is a categorical equivalence between the partial actions of a group $G$ on sets and the star injective functors from groupoids to the group $G$, considered as a groupoid with a single object \cite{KL}. But as groupoids themselves are more general structures then groups, partial actions of groupoids became a subject by itself \cite{BEM}. We can also cite the role of groupoids in the globalization of partial actions of groups \cite{EGG}. The algebraic theory of partial representations of finite groups was particularly well developed because of the isomorphism between  the partial group algebra $k_{par}G$ of a finite group $G$, and the groupoid algebra $k\Gamma (G)$ of a particular groupoid $\Gamma (G)$, whose objects are subsets of the group \cite{dok0}. The very geometric structure of the groupoid $\Gamma (G)$ in terms of its connected  components allows one to decompose the partial group algebra into a product of matrix algebras, in particular, for the case when ${\rm char} (k) \nmid |G|$ it can be proved that the partial group algebra is semi-simple. The groupoid structure also is a tool for classifying irreducible partial representations of an arbitrary group \cite{dok1}.

To point out other noteworthy advances from the purely algebraic point of view, there are several ring theoretic considerations. For example, partial actions of groups on semiprime rings \cite{F} and the study of the structure of partial skew group rings with regard its semisimplicity and semiprimeness \cite{FL}. Another highlight is the development of the theory of twisted partial group actions. In this context, it was possible to characterize a $G$-graded algebra as a crossed product by a twisted partial action of $G$ \cite{DES1}. Also the globalization of twisted partial group actions was developed in reference \cite{DES2}. Regarding globalization, one of the most important recent results can be found in reference \cite{ADES}, in which the notion of Morita equivalence between partial actions was given and it was proved that every partial group action over a not necessarily unital algebra is Morita equivalent to a globalizable (unital) partial action.

The attempt to extend the Galois theory for commutative algebras, due to Chase, Harrison and Rosenberg \cite{CHR}, to the case of partial group actions (see reference \cite{DFP}) spawned a new and unexpected development, namely the extension of the notion of partial actions to the realm of Hopf algebras. There are two seminal articles that triggered the study of partial actions of Hopf algebras. The first, authored by Caenepeel and De Groot \cite{CdG}, extended the Galois theory for partial group actions on noncommutative rings, using the technique of Galois corings \cite{Brz,Caen, Wis}. The second, by Caenepeel and Janssen \cite{CJ}, defined partial entwining structures, as well their duals, partial smash products. From a specific partial entwining structure involving a Hopf algebra $H$ and an algebra $A$, one can define what is a partial coaction of $H$ on $A$ and from a smash product structure involving the same ingredients, we obtain the notion of a partial action of $H$ on $A$. In this paper, the authors discuss the relation between weak and partial entwining structures, they explore dualities and define a partial Hopf-Galois theory. The theory of partial entwining structures can be placed in a more abstract setting, namely, the weak theory of monads, due to B\"ohm \cite{B2}.

These above mentioned articles were the starting point of the theory of partial actions of Hopf algebras. The first result was an extension of the Cohen-Montgomery duality theorem \cite{CM} for partial group actions, but then using Hopf algebra techniques \cite{L}. Afterwards, several developments appeared in the literature, as examples, we can point out the globalization theorem for partial actions of Hopf algebras \cite{AB}, the construction of the subalgebra of partial invariants and a Morita context between the partial smash product and the invariant subalgebra \cite{AB2}. The globalization theorem proved to be an important tool for solving problems involving partial actions. For example, in  \cite{AB3} it was possible to generalize the result of C. Lomp \cite{L} using globalization, proving an analogue of the Blattner-Montgomery duality theorem \cite{BM}
in the case of partial Hopf actions. Going ahead, twisted partial actions of Hopf algebras were defined as well \cite{ABDP1} and their globalization was recently achieved \cite{ABDP2}. Partial (co)actions of Hopf algebras on $k$-linear categories were also studied in \cite{AAB}. New classes of nontrivial examples have been discovered, that are typically Hopf algebraic and not arising from partial group actions, which makes the theory of partial actions of Hopf algebras more independent from the original theory of partial group actions. A deeper connection between partial actions of Hopf algebras and Hopf algebroids \cite{B} was established in reference \cite{ABV}, where the notions of a partial representation of a Hopf algebra and of a partial module over a Hopf algebra were introduced. This new approach, based on Hopf algebroids allowed to introduce a more categorical point of view for partial actions. This was an important step in order to make partial representations an interesting tool for representation theory. The category of partial modules has a monoidal structure which was not observed in the previous works for the group case. Then, even for the study of partial representations of groups, the techniques of monoidal categories coming from the theory of partial representations of Hopf algebras can be very useful. Finally, Hopf algebras are have a rich duality theory. In the context of partial actions of Hopf algebras, several dualities were constructed in \cite{BV}, expanding the universe of interesting objects related to partial actions and coactions.

The aim and scope of this article is to review the basic definitions and constructions related to partial actions of groups and Hopf algebras for a broad mathematical audience. We highlight important examples and results scattered throughout the literature in order to draw the reader's attention to what has been done hitherto in the theory and also to point out some directions for future research. We mainly emphasize the recent developments of the theory of partial actions in the realm of Hopf algebra theory. The basic reason for our choice is the existence of an extensive literature summarizing the advances and the achievements for the case of partial actions and partial representations of groups. Although, it is worthy to be pointed out that our developments in partial (co)actions of Hopf algebras are deeply rooted on what has been done for groups by the above mentioned mathematicians and many others, who were not mentioned but whose importance is widely acknowledge by the mathematical community. All omissions in this paper may be reckoned as an author's fault. For the interested reader, we recommend, for example \cite{D} and \cite{E-4}, and references therein, for a thorough description of techniques and results of partial group actions. 

This paper is organized as follows: In section 2 we review the basic issues about partial group actions. We point out the relationship between partial group actions and groupoids and present the notion of a partial skew group algebra, which is one of the most important constructions in the whole theory. In the same section we show that, for the case of a partial action of a finite group on a finite set, the partial skew group algebra is isomorphic to the algebra of the groupoid associated to that partial action. As the algebra of a finite groupoid is a weak Hopf algebra, this isomorphism endows the partial skew group algebra with a weak Hopf algebra structure. This result, even though it is very simple, did not appear earlier in the literature and it can be considered the only result in this article which was not published elsewhere. In section 3, we define partial actions and partial coactions of Hopf algebras, we present some nontrivial examples of partial actions, as well as the partial smash product, the most important and useful construction in the theory of partial actions of Hopf algebras. We are emphasizing most the partial actions but  we will point out very briefly some dual constructions recently defined. In section 4, we address the several globalization theorems that exist in the context of partial group actions and partial actions of Hopf algebras.  The partial representations of groups and Hopf algebras are treated in section 5. First we review very briefly the relationship between partial actions of groups and inverse semigroups, which provides a broader environment to treat partial group actions. For the linearized case, i.e.\ partial representations of groups on algebras, we describe the universal partial group algebra which factorizes every partial representation by a morphism of algebras. The rich structure of the partial group algebra is described. For partial representations of Hopf algebras, there is also a universal algebra factorizing partial representations by algebra morphisms and this algebra has the structure of a Hopf algebroid. We point out the role played by Hopf algebroids in the whole theory and describe the monoidal structure of the category of partial modules. Finally, in section 6 we mention some open questions and give an outlook to some directions for future research, arguing that partial actions and partial representations have the potential to be a more fundamental tool in representation theory.

\section{Partial actions of groups}

We start by defining the basic notion of a partial action of a group on a set. Partial actions of groups are more common in mathematics than one would expect. Whenever we deal with ordinary differential equations whose flow is incomplete, or when we consider fractional linear transformations, we are dealing with examples of partial actions of groups. This section on partial group actions does not intend to be exhaustive, for more details about partial group actions, see the excellent survey by M. Dokuchaev \cite{D}.

\begin{defi} 
A partial action $\alpha = (\{ X_g \}_{g \in G}, \{\alpha_g \}_{g \in G})$ of a group $G$ on a set $X$ consists of a family indexed by $G$ of subsets $X_g \subseteq X$ and a family of bijections $\alpha_g : X_{g^{-1}} \rightarrow X_g$  for each $g\in G$, satisfying the following conditions:
\begin{enumerate}
\item[(i)] $X_e = X$ and $\alpha_e = \mbox{Id}_X$;
\item[(ii)] $(\alpha_h)^{-1}(X_{g^{-1}} \cap X_h) \subseteq X_{(gh)^{-1}}$, for each $g,h \in G$;
\item[(iii)] $\alpha_g (\alpha_h(x)) = \alpha_{gh}(x)$, for each $x \in (\alpha_h)^{-1}(X_{g^{-1}} \cap X_h)$.
\end{enumerate}

A morphism $\phi$ between two partial actions, $\alpha = (\{ X_g \}_{g \in G}, \{\alpha_g \}_{g \in G})$, and $\beta = (\{ Y_g \}_{g \in G}, \{\beta_g \}_{g \in G})$ of the same group $G$ is a function $\phi :X\rightarrow Y$ such that
\begin{enumerate}
\item[(a)] $\phi (X_g )\subseteq Y_g$, for each $g\in G$.
\item[(b)] For every $x\in X_{g^{-1}}$, $\beta_g (\phi (x))=\phi (\alpha_g (x))$.
\end{enumerate}

We denote the category of partial actions of a group $G$ on sets with morphisms of partial actions by $\underline{Psets}_G$ 
\end{defi}

A few remarks are pertinent about the previous definition: Firstly, the axiom (i) reflects the fact that the neutral element of the group acts as the identity, while the axioms (ii) and (iii) combined say that the composition of bijections indexed by elements of the group, wherever it is defined, is compatible with the operation of the group. Secondly, the axioms (i) and (iii) combined lead to the conclusion that $(\alpha_g )^{-1} =\alpha_{g^{-1}}$. Thirdly, from axioms (ii) and (iii), one can also prove that
\[
\alpha_g (X_{g^{-1}}\cap X_h )=X_g \cap X_{gh},
\] 
for all $g,h\in G$.
It is easy to see that any global action $\alpha :G\rightarrow \mbox{Bij}(X)$, where $\mbox{Bij}(X)$ denotes the group of bijections defined on the set $X$, is a partial action of $G$ on $X$. Also it is quite straightforward that a partial action is global if, and only if, for every $g\in G$ we have $X_g =X$.

The notion of a partial action can be specified to several different contexts. For example, to define a partial action of a (discrete) group $G$ on a topological space $X$, we put the domains $X_g$ to be open subsets of $X$ and the partially defined bijections $\alpha_g$ to be homeomorphisms. For the differentiable case, replace local homeomorphisms for local diffeomorphisms. The case where the group itself is a topological group or a Lie group is a bit more complicated and involves some additional conditions in order to establish that the domains are glued in a continuous/smooth manner. In the case of an action of a group $G$ on an algebra $A$, we impose the domains $A_g$ to be ideals of $A$ and the bijections $\alpha_g :A_{g^{-1}}\rightarrow A_g$ to be algebra isomorphisms. Along this text, we are restricting the discussion to the case of {\em unital partial actions}, that is, partial actions of groups on unital algebras in which the ideals $A_g$ are unital algebras , that is, they are written as $A_g =1_g A$ for $1_g \in A$ being a central idempotent in $A$. Explicitly, as the isomorphisms $\alpha_g$ are unital, we have $\alpha_g (1_{g^{-1}})=1_g$, for each $g\in G$.

\begin{exmp} Consider $G$ to be the additive group of integers, $(\mathbb{Z}, +)$, and define $X$ as the subset of non-negative integers $\mathbb{Z}_+$. Then $G$ acts partially on $X$ by translations. Explicitly, for any $n\geq 0$, we have the domains 
$X_{-n} =\mathbb{Z}_+ $ and $X_{n} =\{ m\in \mathbb{Z}_+ \; | \; m\geq n \}$.
The associated bijections are
$\alpha_n :\ X_{-n} \to X_n$, $\alpha_n(m)=m+n$.
\end{exmp}

\begin{exmp} \label{flux} \cite{AbadieTwo} 
Consider a vectorfield $\mathbf{v}: X\rightarrow TX$ defined on a noncompact manifold $X$. The flow associated to $\mathbf{v}$ defines a partial action of the additive group $(\mathbb{R} ,+)$ on $X$. Indeed, for each $x\in X$ let $\gamma_x :\, ]a_x, b_x [\, \rightarrow X$ be the integral curve of the vector field $\mathbf{v}$ on $x$. By an integral curve, we mean that, 
\begin{enumerate}
\item[(i)] $0\in ]a_x, b_x [$; 
\item[(ii)] $\gamma_x (0)=x$;
\item[(iii)] $\dot{\gamma}_x (t)= \mathbf{v}(\gamma_x (t))$, for $t\in ]a_x , b_x[$;
\item[(iv)] $\gamma_x$ is defined in its maximal interval $]a_x , b_x [$. 
\end{enumerate}
For each $t\in \mathbb{R}$, we have the domain
$X_{-t} =\{ x\in X \; | \; t\in ]a_x , b_x [ \}$
and the partially defined diffeomorphisms
$ \alpha_t :  X_{-t} \to X_t $, $\alpha_t (x)= \gamma_x (t)$.

\end{exmp}

\begin{exmp} \label{mobius} \cite{KL} Consider $G=PSL(2, \mathbb{C})$ the group of complex projective transformations on the complex projective line $\mathbb{CP}^1 \cong \mathbb{S}^2$, and let $X=\mathbb{C}$ be the complex plane. The partial action is given by the fractional linear transformations on the complex plane, the so-called M\"obius transformations. That is, for 
\[
g=\left[ \left( \begin{array}{cc} a & b \\ c & d \end{array} \right) \right] \in PSL(2, \mathbb{C}) ,
\]
we have the domains
$$X_{g^{-1}} =
\begin{cases}
\mathbb{C}\setminus \{-d/c\}&{\rm if~}c\neq 0\\
\mathbb{C}&{\rm if~}c= 0
\end{cases}
~~~~~~{\rm and}~~~~~~
X_g=
\begin{cases}
\mathbb{C}\setminus \{a/c\}&{\rm if~}c\neq 0\\
\mathbb{C}&{\rm if~}c= 0
\end{cases}
$$
 and 
$$ \alpha_g :\ X_{g^{-1}} \to X_g,~~~~~\alpha_g(z) = \frac{az+b}{cz+d}. $$
\end{exmp} 

\begin{exmp} \cite{BV} This example illustrates that partial actions of groups on geometric figures can be sensitive for the internal structure of these figures. Consider the action of the group $G=\mathbb{S}^1 \rtimes \mathbb{Z}_2$, which is the semi-direct product of the rotation group $\mathbb{S}^1$ by the group $\mathbb{Z}_2$, acting on $Y=S_0 \cup S_1 \cup S_{-1}$, which consists of three horizontal unit circles in $\mathbb{R}^3$, defined as
\begin{eqnarray*}
S_0 & = & \{ (x,y,z)\in \mathbb{R}^3 \; | \; x^2 +y^2 =1 , \; z=0 \} \\
S_1 & = & \{ (x,y,z)\in \mathbb{R}^3 \; | \; x^2 +y^2 =1 , \; z=1 \} \\
S_{-1} & = & \{ (x,y,z)\in \mathbb{R}^3 \; | \; x^2 +y^2 =1 , \; z=-1 \} 
\end{eqnarray*}
The initial (global) action of $G$ on $Y$ is given by rotations around the $z$ axis and flips around the $x$ axis:
\[
\beta_{(\theta , j)} (x,y,z) =( x\cos \theta -jy\sin \theta ,x\sin \theta +jy \cos \theta ,jz), \qquad j=\pm 1 .
\]

If we restrict ourselves to the subset $X=S_0 \cup S_1$, then we end up with a partial action of $G$ on $X$, for $j=1$ we have $X_{(\theta ,1)}=S_0 \cup S_1$ but for $j=-1$ we have 
$X_{(\theta , -1 )}=S_0$ and $\alpha_{(\theta ,j)}=\beta_{(\theta ,j)}|_{X_{(-\theta , j)}}$. Geometrically, this means that we can still rotate the two circles but we can flip only the circle $S_0$. Therefore, while a global action confers rigidity to the geometric figure, a partial action decouples its parts, allowing to probe its internal structure.
\end{exmp}

The previous example also explicits how to create, in general, partial group actions out of global actions. In fact, given a (global) group action $\beta :G\rightarrow \mbox{Bij} (Y)$, we can restrict $\beta$ to a partial action on the subset $X\subseteq Y$. Indeed, for each $g\in G$ we put
$X_g =X\cap \beta_g (X)$,
and $\alpha_g :X_{g^{-1}}\rightarrow X_g,\ \alpha_g =\beta_g|_{X_{g^{-1}}}$. Then $\alpha=(\{X_g\}_{g\in G},\{\alpha_g\}_{g\in G})$ defines a partial action  of $G$ on $X$ \cite{dok}. We shall see later in the section about globalization theorems that every partial action of a group on a set is isomorphic to a restriction of a global action of the same group on a larger set.

The restriction of domains which characterizes partial actions of groups on sets entails situations in which the composition of two (partially defined) bijections is not properly defined. This suggests that partial group actions can be related to another important mathematical structure whose operations are not globally defined, namely groupoids. A groupoid is a small category, i.e.\ an internal category in $\underline{Set}$, where all morphisms are isomorphisms. More precisely, a Groupoid is a septuple $(\mathcal{G}, \mathcal{G}^{(0)}, s,t,m,e,(\; )^{-1})$ in which $\mathcal{G}$ and $\mathcal{G}^{(0)}$ are sets, called respectively by the set of arrows and the set of objects, the maps $s$ and $t$ between $\mathcal{G}$ and $\mathcal{G}^{(0)}$, are called respectively source and target maps (intuitively, an element $\gamma \in \mathcal{G}$ can be viewed as a morphism in the category $\mathcal{G}$ between $s( \gamma)$, its domain, and $t(\gamma)$, its range). The map $m: \mathcal{G}^{(2)}\rightarrow \mathcal{G}$ is called multiplication, its domain is the set 
\[
\mathcal{G}^{(2)} =\{ (\gamma , \delta ) \in \mathcal{G} \times \mathcal{G} \; | \; s(\gamma )=t(\delta ) \}
\]
and it is denoted by $m(\gamma , \delta )=\gamma \delta$ (intuitively, it can be viewed as composition of morphisms between $s(\delta )$ and $t(\gamma )$). The multiplication, where it is defined, is associative. The map $e:\mathcal{G}^{(0)} \rightarrow \mathcal{G}$ is called the unit map and associates to each object $x\in \mathcal{G}^{(0)}$ the identity map $e_x \in \mathcal{G}$ whose source and target coincides with $x$. Finally, the map $(\; )^{-1} :\mathcal{G} \rightarrow \mathcal{G}$ is the inversion, which associates to each arrow $\gamma \in \mathcal{G}$ its inverse $\gamma^{-1}$ such that $s(\gamma^{-1})=t(\gamma )$, $t(\gamma^{-1})=s(\gamma )$, $\gamma \gamma^{-1} =e(t(\gamma ))$ and $\gamma^{-1} \gamma =e(s(\gamma ))$. Often, one identifies a groupoid with its set of arrows $\mathcal G$.

To a given partial action $\alpha =( \{ X_g \}_{g\in G} , \{ \alpha_g \}_{g\in G} )$ of a group $G$ on a set $X$, one can associate a groupoid \cite{Aba2}
\[
\mathcal{G} (G,X,\alpha )= \{ (g,x) \in G\times X \; | \; x\in X_{g^{-1}} \} .
\]
$\mathcal{G}^{(0)}=X$ is the set of objects. The source and target maps are defined by the formulas
$s(g,x) =x$ and $t(g,x)=\alpha_g (x)$.
The multiplication is given by
\[
(g,x)(h,y)=
\begin{cases}
(gh, y)&{\rm if}~x=\alpha_h (y)\\
\underline{\quad } &{\rm otherwise}
\end{cases},\]
and the units and inverses are given by the formulas
$e (x)= (e,x)$ and $(g,x)^{-1} =(g^{-1} , \alpha_g (x) )$,
where $e$ is the unit element of $G$..
Note that the projection $\pi_1 :\mathcal{G}(G,X,\alpha )\rightarrow G$, given by $\pi_1 (g,x)=g$ is a functor from the category $\mathcal{G}(G,X,\alpha )$ to the category $G$ viewed as a small category with one object and whose arrows are the elements of the group. This functor happens to be star injective. More precisely, given a small category $\mathcal{C}$, the star over an object $x\in \mathcal{C}^{(0)}$ is the set 
\[
\mbox{Star}(x)=\{ f:x \rightarrow y \; | \; y\in \mathcal{C} \} .
\]
A functor $F:\mathcal{C}\rightarrow \mathcal{D}$, between two small categories $\mathcal{C}$ and $\mathcal{D}$ is called star injective (surjective), if  the map $F|_{\mbox{\small{Star}}(x)} :\mbox{Star}(x) \rightarrow \mbox{Star}(F(x))$ is injective (surjective), for every object $x\in \mathcal{C}^{(0)}$. The star injectivity of the projection map 
$\pi_1 :\mathcal{G}(G,X,\alpha )\rightarrow G$ can be rephrased as $\pi_1$ is star injective if $\pi_1 (\gamma )=\pi_1 (\delta )$ and $s(\gamma )=s (\delta )$, then $\gamma =\delta$ and its star surjectivity, which is equivalent to the action $\alpha$ to be global, can be rephrased as $\pi_1$ is star surjective if for each $g\in G$ and $x\in X$ there is an arrow $\gamma$ such that $\pi_1 (\gamma )=g$ and $s(\gamma )=x$ \cite{KL}.

Conversely, given a groupoid $\mathcal{G}$ and a group $G$ with a star injective functor $F: \mathcal{G} \rightarrow G$ (viewing $G$ as a small category with one object), we can associate a partial action of $G$ on the set of objects $X=\mathbb{G}^{(0)}$. Indeed, for each $g\in G$, we define
\[
X_g =\{ x\in \mathcal{G}^{(0)} \; | \; \exists \gamma \in \mathcal{G}, \; \; s(\gamma )=x , \; F(\gamma )=g^{-1} \} .
\]
and we define $\alpha_g :X_{g^{-1}}\rightarrow X_g$ as
\[
\alpha (x)=t(\gamma ), \quad \mbox{ such that } \quad s(\gamma )=x, \; F(\gamma )=g .
\]
This map is well defined because the functor $F$ is star injective. This defines a partial action $\alpha=(\{X_g\}_{g\in G},\{\alpha_g\}_{g\in G})$ of $G$ on $X$, and this action is global if and only if the functor $F$ is star surjective \cite{KL}. 

These two results above establish a categorical equivalence between the category $\underline{Psets}_G$ of partial actions of the group $G$ and $\underline{\ast Inj}_G$ whose objects are pairs 
$(\mathcal{G}, F:\mathcal{G}\rightarrow G)$ consisting of a groupoid and a star injective functor from this groupoid to the group $G$, and whose  morphisms are functors between groupoids which entwine their respective star injective functors to $G$ \cite{KL}.

The power of the theory of partial actions can be felt better when one considers partial actions on algebras. Given any partial action $\alpha$ of a group $G$ on a set $X$, it is possible to define a partial action of the same group on the algebra of functions of $X$ with values in a field $k$, denoted here by $A=\mbox{Fun}(X,k)$, the domains $A_g$ consist of the functions which vanish outside $X_g$ and the action $\theta_g$ is given by
\begin{equation}
\label{actiononfunctions}
( \theta_g (f)) (x)=f(\alpha_{g^{-1}}(x)) , \qquad \mbox{ for }\; f\in A_{g^{-1}}, \; \mbox{ and } \; x\in X_g.
\end{equation}
The domains $A_g$ above defined are ideals of $A$ and are generated by the idempotents $1_g =\chi_{{}_{X_g}}$ which are the characteristic functions of the domains, therefore, the partial action is unital. For the case of $X$ being a topological space, we require that the partial action is made on the algebra of continuous functions with values in a topological field $k$ (usually $k=\mathbb{R}$ or $k=\mathbb{C}$), and the domains $A_g$ are given by the ideals of continuous functions vanishing outside $X_g$, the difference is that their characteristic functions, in general, are not continuous functions (in fact they are only in very specific topologies, in which the domains are clopen subsets). This produces what is called non unital partial actions, they are very important and far more abundant then the unital ones, but for the purposes of this survey, we will not talk much about them. One reason is that this requires techniques coming from multiplier algebras and its transposition for the Hopf context is still under development. 

One of the standard constructions that arise in the theory of partial group actions on algebras is the partial crossed product, or partial skew group algebra \cite{quigg,dok}. Given a (not necessarily unital) partial action 
$\alpha =(\{ A_g \}_{g \in G}, \{\alpha_g \}_{g \in G})$ of a group $G$ on a (not necessarily unital) algebra $A$, we define a multiplication on the $k$-vector space
\[
A\rtimes_{\alpha} G =\bigoplus_{g\in G} A_g \delta_g
\]
by the following expression
\begin{equation}
\label{produtocruzado}
(a_g \delta_g )(b_h \delta_h )=\alpha_g (\alpha_{g^{-1}} (a_g )b_h )\delta_{gh} .
\end{equation}
Remark that $a_g \in A_g$, which implies that $\alpha_{g^{-1}} (a_g )\in A_{g^{-1}}$. As $A_{g^{-1}}$ and $A_h$ are ideals, then the product $\alpha_{g^{-1}} (a_g )b_h$ lies in the intersection $A_{g^{-1}} \cap A_h$ and by the axiom (ii) we know that $\alpha_g (\alpha_{g^{-1}} (a_g )b_h ) \in A_{gh}$. Therefore, the expression (\ref{produtocruzado}) is well defined. The product, a priori is not associative, the most general case in which we can ensure the associativity is when the algebra $A$ is semiprime, that is, it has no nilpotent ideals \cite{dok}. For the case of unital actions the expression (\ref{produtocruzado}) can be written in a simpler way, namely
\[
(a_g \delta_g )(b_h \delta_h ) =a_g \alpha_g (b_h 1_{g^{-1}}) \delta_{gh} ,
\] 
and the partial skew group algebra turns out to be unital, with
$1_{A\rtimes_{\alpha} G} =1_A \delta_e$.

In the case of unital partial actions, the sufficient conditions for the associativity of the partial crossed product are automatically satisfied.

The role played by the partial skew group algebra is fundamental for the development of several aspects of the theory. As we shall see later, it plays a central role in globalization, Galois theory for partial actions, and partial representations.

Just to summarize what has been exposed in this section, let us relate the algebra of the groupoid of a partial action with the partial skew group algebra. Consider a partial action $\alpha$ of a finite group $G$ on a finite set $X$. We just have seen that $\alpha$ induces a partial action $\theta$ of $G$ on the algebra $A=\mbox{Fun}(X, k)$ given by (\ref{actiononfunctions}). Also we have the groupoid of a partial action $\mathcal{G}=\mathcal{G}(G,X,\alpha)$. Let us write the elements of $\mathcal{G}$ in a slightly different, although equivalent form:
\[
\mathcal{G}(G,X,\alpha)=\{ (x,g)\in X\times G \;| \; x\in X_g \} .
\]
The groupoid algebra $k\mathcal{G}(G,X,\alpha)$ is generated as a $k$-vector space by the basis elements $\delta_{(x,g)}$ whose product is given by 
\[
\delta_{(x,g)}.\delta_{(y,h)}=\delta_{(x,gh)} [\! [ \alpha_{g^{-1}} (x)=y ]\! ] , 
\]
in which $[\! [ \quad ]\! ]$ denotes the boolean function which takes value $1$ when the sentence between brackets is true and the value $0$ when the sentence between brackets is false, and the unit is given by
\[
\mathbf{1}=\sum_{x\in X} \delta_{(x,e)} .
\]
Then we have the following result:

\begin{thm} \label{partialskewisgroupoid} Given a partial action $\alpha$ of a finite group $G$ on a finite set $X$, the algebra of the partial action groupoid $k\mathcal{G}(G,X,\alpha)$ is isomorphic to the partial skew group algebra $\mbox{Fun}(X,k)\rtimes_{\theta} G$, in which the partial action $\theta$ of $G$ on the algebra of functions $\mbox{Fun}(X,k)$ is that induced by the partial action $\alpha$.
\end{thm}

\begin{proof} Define two maps $\Phi :\mbox{Fun}(X,k)\rtimes_{\theta} G \rightarrow k\mathcal{G}(G,X,\alpha)$ and $\Psi : k\mathcal{G}(G,X,\alpha) \rightarrow \mbox{Fun}(X,k)\rtimes_{\theta} G $ using the formulas
\[
\Phi (a\delta_g )=\sum_{x\in X_g }a(x) \delta_{(x,g)} , \qquad \Psi (\delta_{(x,g)})=\chi_x \delta_g. 
\]
Here $\chi_x$ is the characteristic function on $x\in X_g$, that is, $\chi_x (y)=[\! [ x=y ]\! ]$. It is easy to see that $\Phi =\Psi^{-1}$:
\begin{eqnarray*}
\Phi (\Psi (\delta_{(x,g)}))&=& \Phi (\chi_x \delta_g ) =\sum_{y\in X_g} \chi_x (y) \delta_{(y,g)}=\delta_{(x,g)};\\
\Psi (\Phi (a\delta_g ))&=&\sum_{x\in X_g} a(x) \Psi (\delta_{(x,g)})=\sum_{x\in X_g} a(x)\chi_x \delta_g =a\delta_g .
\end{eqnarray*}
$\Phi$ is a morphism of algebras since
\begin{eqnarray*}
&&\hspace*{-20mm}
\Phi ((a\delta_g )(b\delta_h ) = \Phi (\theta_g (\theta_{g^{-1}} (a)b) \delta_{gh})
= \sum_{x\in X_{gh}} \theta_g (\theta_{g^{-1}} (a)b)(x) \delta_{(x, gh)} \\
& = & \sum_{x\in X_{gh}} (\theta_{g^{-1}} (a))(\alpha_{g^{-1}}(x)) b(\alpha_{g^{-1}}(x)) \delta_{(x, gh)} \\
&=& \sum_{x\in X_{gh}} a(x) b(\alpha_{g^{-1}} (x)) [\! [ \alpha_{g^{-1}} (x)\in X_h ]\! ] \delta_{(x,gh)} \\
& = & \sum_{x\in X_g \cap X_{gh}} a(x) b(\alpha_{g^{-1}} (x))\delta_{(x,gh)}\\
&=&\sum_{x\in X_g}\sum_{y\in X_h} a(x)b(y) [\! [ \alpha_{g^{-1}} (x)=y ]\! ] \delta_{(x,gh)} \\
& = & \sum_{x\in X_g}\sum_{y\in X_h} a(x)b(y) \delta_{(x,g)} \delta_{(y,h)}
= \Phi (a\delta_g )\Phi (b\delta_h) .
\end{eqnarray*}
\end{proof}

An immediate consequence is that the groupoid algebra $k\mathcal{G}(G,X,\alpha)$, is a weak Hopf algebra, with structure
\[
\Delta (\delta_{(x,g)})=\delta_{(x,g)} \otimes \delta_{(x,g)}, \quad \epsilon_{\delta_{(x,g)}}=1, \quad S(\delta_{(x,g)})=\delta_{(\alpha_{g^{-1}}(x), g^{-1})} ,
\]
the partial skew group algebra is endowed with a structure of a weak Hopf algebra, with structure
\[
\Delta (a\delta_g )=\sum_{x\in X_g} a(x)\chi_x \delta_g \otimes \chi_x \delta_g , \quad \epsilon (a\delta_g ) =\sum_{x\in X_g} a(x), \quad S(a\delta_g )=\theta_{g^{-1}}(a) \delta_{g^{-1}} .
\]

\section{Partial (co)actions of Hopf algebras}

The most remote sources of partial actions of Hopf algebras can be found in the theory of partial Galois extensions, which was a generalization of the Galois theory for commutative rings by Chase, Harrison and Rosenberg to the case of partial group actions.  In order to be more precise, let us recall some concepts involved in the Galois theory for partial actions. Given a commutative ring $k$ and a unital partial action $\alpha$, of a group $G$ on a $k$-algebra $A$, we define the partial invariant subalgebra as
\[
A^{\alpha} =\{ a\in A \; | \; \alpha_g (a1_{g^{-1}})=a1_g \; \} .
\]

\begin{defi} \cite{DFP} Let $k$ be a commutative ring and $A$ and $B$ be $k$-algebras and $\alpha$ is a unital partial action of a finite group $G$ on $A$. The algebra $A$ is said to be a partial Galois extension of $B$ if
\begin{enumerate}
\item[(i)] $A^{\alpha} =B$.
\item[(ii)] There exist elements $x_i , y_i \in A$, $1\leq i \leq n$ such that $\sum_{i=1}^n x_i \alpha_g (y_i 1_{g^{-1}} )=\delta_{g,e}$, for each $g\in G$.
\end{enumerate}
\end{defi}

Actually, there are several equivalent conditions to say that $A$ is a partial Galois extension of $B$ (see \cite[Theorem 4.1]{DFP}. Moreover, the concept of partial Galois extension can be rephrased in terms of Galois corings. Consider a unital partial action of a finite group $G$ on an algebra $A$. Out of this data one can construct an $A$-coring 
\cite{CdG} 
\[
\mathcal{C} =\bigoplus_{g\in G} A_g v_g = \bigoplus_{g\in G} (A1_g) v_g ,
\] 
whose $A$-bimodule structure is given by
\[
a(b v_g )a' =ab \alpha_g (a' 1_{g^{-1}}) ,
\]
and whose comultiplication and counit are given, respectively, by
\[
\Delta_{\mathcal{C}} (av_g) =\sum_{h\in G} av_h \otimes_A v_{h^{-1}g} , \qquad \epsilon_{\mathcal{C}} (av_g) = a\delta_{g,e} .
\]
Note that in the case of non-unital partial actions we cannot define the right $A$ module structure on $\mathcal{C}$. Also, the condition of $G$ being finite is necessary in order to define the comultiplication above. In the coring $\mathcal{C}$, one can prove that the element $x=\sum_{g\in G} v_g$ is grouplike and the condition of $A^{\alpha}\subset A$ being a partial Galois extension is equivalent to say that the canonical map
\[\mbox{can}:\ A\otimes_{A^{\alpha}} A\to \mathcal{C},~~\mbox{can}(a\otimes b)= \sum_{g\in G} a\alpha_g (b1_{g^{-1}}) v_g.\]
is an isomorphism. In fact, It is shown in \cite{CdG} that the classical results of Hopf-Galois theory and its relation to Morita theory are
still valid in the context of partial actions. 

The turning point in the theory was given in reference \cite{CJ}, in which coring techniques were strongly used to extend the notion of partial actions to the Hopf algebraic context. The underlying construction behind partial coactions of Hopf algebras are the so-called partial entwining structures, while their dual notion, the partial smash product structures, leads to the concept of a partial action of a Hopf algebra. For our purposes, we will define partial (co)actions of Hopf algebras directly in the form they are used currently in the literature.

\begin{defi} Let $k$ be a commutative ring, $H$ be a $k$-bialgebra and $A$ a unital $k$-algebra.
A left partial action of $H$ on $A$ is a linear map
$$ \cdot :\ H\otimes A \to A ,~~h\otimes a  \mapsto  h\cdot a,$$
satisfying the conditions
\begin{enumerate}
\item[(PA1)] $1_H \cdot a =a$, for all $a\in A$;
\item[(PA2)] $h\cdot (ab) =(h_{(1)}\cdot a)(h_{(2)}\cdot b)$, for all $h\in H$ and $a,b\in A$;
\item[(PA3)] $h\cdot (k\cdot a)=(h_{(1)}\cdot 1_A )(h_{(2)}k\cdot a)$, for all $h,k\in H$ and $a\in A$.
\end{enumerate}
The partial action is symmetric if, in addition, we have
\begin{enumerate}
\item[(PA3')] $h\cdot (k\cdot a)=(h_{(1)}k\cdot a)(h_{(2)}\cdot 1_A )$, for all $h,k\in H$ and $a\in A$.
\end{enumerate}
$A$ is called a left partial $H$-module algebra. The definition of a right (symmetric) partial action of $H$ on $A$ is completely analogous.\\
A right partial coaction of $H$ on $A$ is a linear map
$$ \rho :\\ A \to A\otimes H,~~\rho(a)= a^{[0]}\otimes a^{[1]},$$
satisfying the contions
\begin{enumerate}
\item[(PC1)] $(Id_A \otimes \epsilon ) \circ \rho (a) =a$, for all $a\in A$;
\item[(PC2)] $\rho(ab) =\rho(a)\rho (b)$, for all $a,b\in A$;
\item[(PC3)] $(\rho \otimes Id_H) \circ \rho(a)=\left( (Id_A \otimes \Delta )\circ \rho (a) \right) (\rho(1_A ) \otimes 1_H)$, for all $a\in A$.
\end{enumerate}
The partial coaction is symmetric if, in addition, we have
\begin{enumerate}
\item[(PC3')] $(\rho \otimes Id_H) \circ \rho(a)= (\rho(1_A ) \otimes 1_H)\left( (Id_A \otimes \Delta )\circ \rho (a) \right) $ for all $a\in A$.
\end{enumerate}
We say that $A$ is a right partial $H$-comodule algebra. The definition of a left (symmetric) partial coaction of $H$ on $A$ is completely analogous.
\end{defi}

\begin{remarks}
1)
Remark that a left $H$-module algebra $A$, tis automatically a left symmetric partial $H$-module algebra. In fact, a left partial 
$H$-module algebra is a left $H$-module algebra if and only if $h\cdot 1_A =\epsilon (h) 1_A$, for all $h\in H$. Analogously, a right $H$-comodule algebra,is automatically a right symmetric partial $H$-comodule algebra. A right partial $H$-comodule algebra is a right $H$-comodule algebra if and only if $\rho (1_A ) =1_A \otimes 1_A$. \\
2) A genuine partial $H$-module algebra is not an $H$-module, because the axiom (PA3) destroys  the associativity required to be a module. A similar observation holds for a partial $H$-comodule algebra, because the axiom (PC3) destroys the coassociativity of the definition of an $H$-comodule.\\
3) The axiom (PA2) leads to the conclusion that the map $e:\ H\rightarrow A$, which associates to each $h\in H$ the element $h\cdot 1_A \in A$, is an idempotent in the convolution algebra $\mbox{Hom}_k (H,A)$. This is very useful to obtain many results in the theory. Similarly, the axiom (PC2) leads to the conclusion that $\rho (1_A )$ is an idempotent in $A\otimes H$, therefore, the image of $\rho$ lies in a direct summand of the $k$-module $A\otimes H$. \\
4) We only deal with symmetric partial (co)actions because those are required for proving most of the important results. We will mainly focus on partial actions. Nevertheless, there are nice results involving partial coactions of Hopf algebras which will be commented briefly along this paper. IIn general, we restrict to the case where $k$ is a field, although most of the results can be extended to the case when $k$ is a commutative ring. 
\end{remarks}

Partial actions and partial coactions are dually related \cite{AB, AB3}. Let $H$ be a Hopf algebra and  $A$ a right partial $H$-comodule algebra with partial coaction $\rho :A\rightarrow A\otimes H$, then $A$ can be endowed with the structure of a left partial $H^\circ$-module algebra, with $H^\circ$ denoting the finite dual of $H$. Explicitly, the partial $H^\circ$ action is given by
\begin{equation}\label{duality}
h^* \cdot a= a^{[0]} h^* (a^{[1]}), 
\end{equation}
for all $a\in A$ $h^* \in H^\circ$ and $\rho (a)=a^{[0]}\otimes a^{[1]}\in A\otimes H$. Dually, if $A$ is a left partial $H$-module algebra such that for each $a\in A$ the subspace $H\cdot a$ is finite dimensional, that is there exists $a_i \in A$ and $h^*_i \in H^\circ$, for $i\in \{ 1, \ldots n \}$ such that 
\[
h\cdot a=\sum_{i=1}^n h^*_i (h) a_i ,
\]
then, $A$ is a right partial $H^\circ $ comodule algebra with partial coaction
\[
\rho (a)= \sum_{i=1}^n a_i \otimes h^*_i .
\] 

\begin{exmp} \label{unitalpartialaction} \cite{CJ} As a basic example, to make a bridge with the previous section, consider  a unital partial action $\alpha$ of a group $G$ on a
unital algebra $A$. Recall that a unital partial action means that the ideals $A_g$  are of the form $a_g =A1_g$ in which $1_g$ is a central idempotent in $A$ for each $g\in G$. This defines a partial action of the group algebra $kG$ on $A$, which is defined on the basis elements by
\begin{equation}
\label{actiongroupalgebra}
\delta_{g} \cdot a =\alpha_g (a 1_{g^{-1}}),
\end{equation}
and extended linearly to all elements of $kG$. 

Note that axiom (PA1) translates into the fact that the neutral element of the group acts as the identity oon $A$, the axiom  (PA2) reflects the fact that each $\alpha_g : A_{g^{-1}}\rightarrow A_g$ is a unital isomorphism and (PA3) can be rewritten as
\[
\delta_g \cdot (\delta_h \cdot a)=1_g (\delta_{gh} \cdot a)=(\delta_g \cdot 1_A )(\delta_{gh} \cdot a)=(\delta_{gh} \cdot a)(\delta_g \cdot 1_A ).
\]
Therefore, the partial action is symmetric because the idempotents $1_g$ are central.
\end{exmp}

The connection between unital partial actions of a group $G$ and the corresponding partial actions of the group algebra $kG$ allows one to see two different points of view concerning partiality. When we consider partial group actions, we end up with partially defined isomorphisms on restricted domains inside the whole algebra, while when we look at the partial action from the Hopf algebra perspective, what is a watermark of partiality is the lack of the structure of module. This discussion will be resumed later, when we are dealing with partial representations and partial modules.

\begin{exmp} \cite{AB} Let $H$ be a bialgebra, $B$ be a left $H$-module algebra and $e=e^2 \in B$ be an idempotent central in $B$. One can define a partial action of $H$ on the ideal $A=eB$ by
\[
h\cdot (eb)=e(h\triangleright (eb)), \qquad \mbox{ for  } b\in B .
\]
Actually, as we shall see later, the globalization theorem states that every partial acion of a Hopf algebra is of this type.
\end{exmp}

\begin{exmp} Let $\mathfrak{g}$ be a Lie algebra, then every partial action of its universal enveloping algebra, $\mathcal{U}(\mathfrak{g})$ is global \cite{CJ}.
\end{exmp}

\begin{exmp} \cite{AAB} Consider a field $k$ and a $k$-Hopf algebra $H$, a symmetric partial action of $H$ on the base field $k$ is equivalent to a linear functional $\lambda :H\rightarrow k$ satisfying
\begin{enumerate}
\item[(i)] $\lambda (1_H )=1$;
\item[(ii)] $\lambda (h)= \lambda (h_{(1)}) \lambda (h_{(2)})$;
\item[(iii)] $ \lambda (h) \lambda (k) =\lambda (h_{(1)}) \lambda (h_{(2)}k)=\lambda (h_{(1)} k) \lambda (h_{(2)})$.
\end{enumerate}

For the group algebra $kG$, denoting $\lambda (\delta_g )$ simply by $\lambda_g$, for every $g\in G$, we have the conditions
\[
\lambda_e =1, \qquad \lambda_g =\lambda_g \lambda_g , \qquad \lambda_g \lambda_h =\lambda_{gh} \lambda_g .
\]
The second equality tells us that the only possibilities are $\lambda_g =0$ or $\lambda_g =1$. Consider the following subset
\[
H=\{ g\in G \; | \; \lambda_g =1 \} .
\]
One can easily check that $H$ is a subgroup of $G$. Therefore, partial actions of $kG$ on the base field $k$ are parametrized by subgroups of $G$.

For a finite group $G$, a partial $G$-grading on an algebra $A$ is by definition a symmetric partial action of the dual group algebra $(kG)^*$ on $A$. In particular, a partial $G$-grading on the base field $k$ is defined as a  symmetric partial action of $(kG)^*$ on $k$. The Hopf algebra $(kG)^*$ is the dual of the group algebra of $G$ which is generated as algebra by the complete set of orthogonal idempotents $\{ p_g \}_{g\in G}$ and whose comultiplication, counit and antipode are given, respectively, by
\[
\Delta (p_g )=\sum_{h\in G} p_h \otimes p_{h^{-1}g} =\sum_{h\in G} p_{gh^{-1}} \otimes p_h, \quad \epsilon (p_g )=\delta_{g,e}, \quad S(p_g) =p_{g^{-1}}.
\]
In terms of the functional $\lambda$ the partial action has the following relations:
\begin{eqnarray*}
\sum_{g\in G} \lambda (p_g)  &= & 1,\\
\lambda (p_g ) & = & \sum_{h\in G} \lambda (p_h )\lambda (p_{h^{-1}}g) =\sum_{h\in G} \lambda (p_{gh^{-1}})\lambda (p_h ), \\
\lambda (p_g ) \lambda (p_h ) & = & \lambda (p_{gh^{-1}})\lambda (p_h )=\lambda (p_h )\lambda (p_{h^{-1}g}) .
\end{eqnarray*}
Define 
\[
H=\{g\in G \; | \; \lambda (p_g ) \neq 0 \},
\]
then we can prove that $H$ is a subgroup of $G$ and $\lambda (p_g )=\lambda (p_e )$ for every $g\in H$. This leads to $\lambda (p_g )=\frac{1}{|H|}$ for $g\in H$ and $\lambda (p_g )=0$ for $g\notin H$.

As one more example of this type, consider the Sweedler's four-dimensional algebra $H_4 =\langle 1,g,x \; | \; g^2 =1, \; x^2 =0, \; gx=-xg \rangle$, with Hopf structure given by
\[
\Delta (g)=g\otimes g, \quad \Delta (x)=x\otimes 1 +g\otimes x, \quad \epsilon (g)=1, \quad \epsilon (x)=0, \quad S(g)=g, \quad S(x)=-gx.
\]
In order to classify the partial actions of $H_4$ on the base field $k$ we end up with the following relations on the functional $\lambda$:
\[
\lambda (g)=\lambda (g)\lambda (g),\qquad  \lambda (g)\lambda (x)=\lambda (g) \lambda (gx)=0 , \qquad \lambda (gx ) \lambda (gx )=\lambda (x) \lambda (gx ).
\]
The first equation gives us two possibilities: Either $\lambda (g) =1$ and consequently $\lambda (x)=\lambda (gx)=0$, which is in fact a global action, or $\lambda (g)=0$ and $\lambda (x)=\lambda (gx )=\lambda$, these are all possibilities of partial $H_4$-actions on $k$.
\end{exmp}

Several results on partial $G$-gradings, or, equivalently, partial $(kG)^*$-actions, are presented in \cite{AB,ABV}. Theorem \ref{z2grad} gives us a classification of partially $\mathbb{Z}_2$-graded algebras. 

\begin{thm} \label{z2grad} \cite{ABV} 
Let $k$ be a field of characteristic different from $2$, let $C_2$ be the cyclic group $C_2 = \langle e,g \, |\,
g^2 =e \rangle$ and let $H$ be the dual group algebra $(k C_2)^*$. Let $A$ be a partially $\mathbb{Z}_2$-graded algebra, i.e.\ a partial $(k C_2)^*$-module algebra, then $A$ can be decomposed as a vector space,
\[
A = A_0 \oplus A_1 \oplus A_{\frac{1}{2}},
\]
such that
\begin{enumerate}
\item $ B = A_0 \oplus A_1$ is a subalgebra of $A$ which is $\mathbb{Z}_2$-graded by  $B_{\overline{0}} = A_0$ and $B_{\overline{1}} = A_1$;
\item  $A_{\frac{1}{2}}$ is an ideal of $A$;
\item the unit of $A$ decomposes as  $1_A = u + v$ where  $u,v$ are orthogonal idempotents, $u \in A_0$ and $v \in A_{\frac{1}{2}}$;
\item $uAv = vAu=0$ and therefore $A \simeq B \times A_{\frac{1}{2}}$ as an algebra.  
\end{enumerate}
\end{thm}

\begin{exmp}In reference \cite{ABDP1}, a nontrivial example of a (twisted) partial action coming from algebraic groups was introduced. Let $k $ be an isomorphic copy of the complex numbers $\mathbb{C}$ and let $\mathbb{S}^1  \subseteq \mathbb{C}$ be the circle group, i.e. the group of all complex numbers of modulus one.   Let, furthermore, $G$ be an arbitrary finite group seen as a subgroup of $S_n$ for some $n$. Taking the action of  $G \subseteq S_n$  on $ (k \mathbb{S}^1 )^{\otimes n} $  by permutation of roots,  consider  the smash product Hopf algebra
\[
H'_1 = (k \mathbb{S}^1 )^{\otimes n} \rtimes k  G,
\]
which is co-commutative. Let $X\subseteq G$ be an arbitrary subset which is not a subgroup and consider the subalgebra $\tilde {A} =  (\sum _{g\in X} p_g ) (k G)^* \subseteq  (k  G)^*,$ and define the commutative algebra $A'=  k[t,t^{-1}]^{\otimes n} \otimes \tilde{A}$.

In order to  simplify the notation, write
\begin{equation}\label{notation1}
t_i = 1 \otimes \ldots \otimes 1 \otimes t \otimes 1\otimes \ldots \otimes 1,
\end{equation}
where $t$ belongs to the $i$-copy of $k[ t, t^{-1}]$, then, we have the elementary monomials in $k [ t, t^{-1}]$, given by
\[
t_1^{k_1}\ldots t_n^{k_n} = t^{k_1} \otimes \ldots \otimes t^{k_n} .
\]
In its turn, the generators of $(k \mathbb{S}^1 )^{\otimes n}$ can be written in terms of the roots  of unit $\chi_{\theta}$ in the following way
\begin{equation}\label{notation2}
\chi_{\theta_1 , \ldots \theta_n } = \chi_{\theta_1} \otimes \ldots \otimes \chi_{\theta_n } \in (k \mathbb{S}^1 )^{\otimes n},
\end{equation}
where $\chi_{\theta_i} \in \mathbb{S}^1 $ is the root of $1$ whose angular coordinate is $\theta _i $ and which belongs to the  
$i$th-factor of $(k \mathbb{S}^1 )^{\otimes n}$.

With notation established in (\ref{notation1}-\ref{notation2}), the formula
\begin{eqnarray*}
&&\hspace*{-2cm}
(\chi_{\theta_1 , \ldots \theta_n } \otimes u_g) \cdot   (t_1^{k_1} \ldots t_n^{k_n} \otimes p_s )  \\
&=&
\begin{cases}
\exp\{ i \sum_{j=1}^{n}k_j \theta_{  g s^{-1}  (j)}  \}  \; t_1^{k_1} \ldots t_n^{k_n} \otimes p_{ s g^{-1}   }
&{\rm if~}s^{-1} g \in X\\
0&{\rm if~}s^{-1} g \not\in X,
\end{cases}
\end{eqnarray*}
where $g \in G$ and $s \in X \subseteq G,$  gives an ordinary left partial action  $\cdot : H'_1 \times A' \to A' $. From this partial action, one can obtain a twisted partial action in the case of finite groups $G$ whose Schur multiplier over $\mathbb{C}$ is nontrivial. In order to define a partial two cocycle for the partial action of $H'_1$ over $A'$, one needs first a normalized two cocycle $\gamma: G\times G \rightarrow \mathbb{C}^*$ in the classical Schur multiplier of $G$. The details of this construction can be found in \cite{ABDP1}
\end{exmp}

More nontrivial examples of partial actions and twisted partial actions of Hopf algebras can be found in \cite{AAB, AB3,ABDP1, ABDP2}.

One of the most important constructions related to partial actions of Hopf algebras is the partial smash product. Given a partial action $\cdot :H\otimes A \rightarrow A$ of a bialgebra $H$ on an algebra $A$, one can endow the tensor product $A\otimes H$ with an associative product,
\[
(a\otimes h)(b\otimes k)=a(h_{(1)}\cdot b)\otimes h_{(2)}k .
\]
However, this algebra is not unital, the element $1_A \otimes 1_H$ works as a left unit, but not a right one. Thus, we have to restrict to the subspace
\[
\underline{A\# H }=(A\otimes H)(1_A \otimes 1_H)=\mbox{span} \{ a\# h =a(h_{(1)}\cdot 1_A )\otimes h_{(2)} \; | \; a\in A , \; h\in H \} .
\]
It is easy to see that this algebra will be associative and unital. For the case of $H$ being the group algebra $kG$, the partial smash product coincides with the partial skew group algebra presented in previous section. The role played by the partial smash product in the theory can be better grasped in the following two sections.

We mention some other constructions related to partial (co)actions of Hopf algebras without going into details. 
\begin{enumerate}
\item The twisted partial actions of Hopf algebras by partial 2-cocycles were treated in \cite{ABDP1}. Twisted partial actions lead to a generalization of the partial smash product, called the partial crossed product. Again, symmetric versions of these twisted partial actions are better behaved and from them it is possible to derive many properties, for example, one can extend the concept of a cleft extension for the partial case and we can prove that these partially cleft extensions are in one-to-one correspondence with partial crossed products. Lately, the partial crossed products were placed in a broader context in reference \cite{FRR}, there, the authors introduced the notion of a weak crossed product and showed that the partial crossed products obtained in \cite{ABDP1} were special cases of their construction. 
The globalization of twisted partial actions of Hopf algebras was treated in \cite{ABDP2}.

\item Partial actions of Hopf algebras on $k$-linear categories appeared in \cite{AAB}. There the concept of a partial smash product of a category by a Hopf algebra appeared. Many nontrivial examples were given, in particular the question about partial $G$-gradings first took place in this context. The globalization of a partial action of a Hopf algebra on a $k$-linear category was done in the same paper and the same results of \cite{AB} continued valid for categories.

\item The partial versions of an $H$-module coalgebra and $H$-comodule coalgebra were presented in \cite{BV}, even though we acknowledge that a first version of partial comodule coalgebras first appeared in \cite{wang}. A partial action of a cocommutative Hopf algebra $H$ on a commutative algebra endows the partial smash product $\underline{A\# H}$ with the structure of a Hopf algebroid over $A$. Similarly, given a right partial coaction of a Hopf algebra $H$ over an algebra $A$ defines a canonical $A$ coring 
\[
\underline{A\otimes H}=(A\otimes H )\rho (1_A ),
\]
this is well known from the theory of partial entwining structures \cite{CJ}. In the case when both $H$ and $A$ are commutative, then $\underline{A\otimes H}$ becomes a commutative Hopf algebroid. This relationship between partial coactions of commutative Hopf algebras on commutative algebras and commutative Hopf algebroids can be rephrased in a dual version of the Kellendonk-Lawson's theorem relating partial actions of groups on sets and groupoids \cite{KL}. Finally, given a left partial $H$-comodule coalgebra $C$, one can define a new coalgebra $\underline{C\blacktriangleright \! \! < H}$, the partial cosmash coproduct, which can be seen as a dual of the partial smash product \cite{BV}
\end{enumerate}

\section{Globalization}

As we have seen in section 2, one can create a partial action of a group $G$ out of a global one by restricting to a subset. More precisely, given an action $\beta :G\rightarrow \mbox{Bij}(Y)$ and a subset $X\subseteq Y$, for each $g\in G$ define the subsets
\[
X_g =X\cap \beta_g (X)
\]
and the maps $\widehat{\beta}_g =\beta_g|_{{}_{X_{g^{-1}}}}$, these data define a partial action of $G$ on $X$.  Conversely, given a partial action $\alpha =( \{ X_g \}_{g\in G} , \{ \alpha_g :X_{g^{-1}}\rightarrow X_g \}_{g\in G} )$, one may ask whether is it possible to obtain $\alpha$ from a global action on a bigger set by restriction. This is the problem of globalization. Of course there can be many globalizations for the same partial action, then we have to put a minimality condition in order to obtain a more precise definition.

\begin{defi} \label{globset} A globalization of a partial action $\alpha$ of a group $G$ on a set $X$ is a triple $(Y, \beta, \varphi )$, in which
\begin{enumerate}
\item[(GL1)] $Y$ is a set and $\beta :G\rightarrow \mbox{Bij}(Y)$ is an action of $G$ on $Y$;
\item[(GL2)] $\varphi :X\rightarrow Y$ is an injective map;
\item[(GL3)] for each $g\in G$, $\varphi (X_g )=\varphi (X)\cap \beta_g (\varphi (X))$;
\item[(GL4)] for each $x\in X_{g^{-1}}$, we have $\varphi (\alpha_g (x)) =\beta_g (\varphi (x))$.
\end{enumerate}
A globalization is said to be admissible if
\begin{enumerate}
\item[(GL5)] $Y=\cup_{g\in G} \beta_g (\varphi (X))$.
\end{enumerate}
\end{defi}

Let us make some considerations which are pertinent about the previous definition. First, the axioms (GL3) and (GL4) tell us that the map $\varphi$ is a morphism in $\underline{Psets}_G$ between the partial action $\alpha$ on $X$ and the partial action $\widehat{\beta}$ on $\varphi (X)\subseteq Y$ induced by the global action $\beta$. Second, as the map $\varphi$ is injective by axiom (GL2), in fact we have an isomorphism between these two partial actions. Finally, the admissibility axiom (GL5) imposes a minimality condition on the globalization, this is required in order to garantee uniqueness of the globalization up to isomorphism. 

The first globalization theorems appeared in F. Abadie's PhD thesis \cite{Abadie} and afterwards in reference \cite{AbadieTwo}, of the same author. 

\begin{thm} \cite{AbadieTwo} There exists an admissible globalization for any given partial action $\alpha$ of a group $G$ on a set $X$.
\end{thm}

\begin{proof} The idea of the construction of the globalization is to establish an equivalence relation on the set $G\times X$,
namely 
$(g,x)\sim (h,y)$ if and only if $x\in X_{g^{-1}h}$ and  $\alpha_{h^{-1}g}(x)=y$.
The quotient $X^e =(G\times X)/\sim$ is the total space and the global action is given by the formula
$\beta_g [(h,x)]=[(gh,x)]$, 
where $[(h,x)]$ is the class of the pair $(h,x)$ in $X^e$. It is easy to see that this action is well defined. The inclusion map 
$\varphi :X\rightarrow X^e$ is given by $\varphi (x)=[(e,x)]$, where $e$ denotes the neutral element of $G$. The verification of the axioms (GL1-GL5) in Definition \ref{globset} is straightforward.
\end{proof}

Some difficulties arise when we are dealing with partial actions with extra conditions, like partial actions of groups on topological spaces or partial actions of groups on algebras. For the topological case, the first problem arises when we consider the globalized space $X^e$. As it is a quotient, even when we start with a partial action on a Hausdorff space, in general we end up with a non Hausdorff one. For the example of the partial action of $PSL(2, \mathbb{C})$ on the complex plane by M\"obius transformations, Kellendonk and Lawson proved that the globalized space with its quotient topology coincides with the one point compactification of the complex plane, the Riemann sphere \cite{KL}, as anyone could expect. Now, consider the the example \ref{flux}, describing the flow of a vector field on a differentiable manifold $M$, which corresponds to a partial action of the additive group $(\mathbb{R}, +)$ on $M$. The globalization in this case would produce a non-Hausdorff manifold. In order to characterize the new global action as a complete flow of a new vector field on a manifold, perhaps some techniques coming from noncommutative geometry might be needed. Another form to overcome the non-Hausdorffness of the globalized space is to consider the groupoid of the equivalence relation as a subset of $(G\times X)\times (G\times X)$ instead of considering the quotient $(G\times X)/\sim$ \cite{EGG}. This approach has some advantages, but it requires some techniques coming from groupoids in order to analyse the globalization.

For the case of partial actions of groups on algebras, Definition \ref{globset} of a globalization needs some modifications.

\begin{defi} \label{globalg} \cite{dok} An enveloping action (globalization) of a partial action $\alpha$ of a group $G$ on an algebra $A$ is a triple $(B, \beta, \varphi )$, in which
\begin{enumerate}
\item[(GL1)] $B$ is an algebra, not necessarily unital, and $\beta :G\rightarrow \mbox{Aut}(B)$ is an action of $G$ on $B$ by automorphisms;
\item[(GL2)] $\varphi :A\rightarrow B$ is an algebra monomorphism and $\varphi (A)$ is an ideal in $B$;
\item[(GL3)] for each $g\in G$, $\varphi (A_g )=\varphi (A)\cap \beta_g (\varphi (A))$;
\item[(GL4)] for each $x\in A_{g^{-1}}$, we have $\varphi (\alpha_g (x)) =\beta_g (\varphi (x))$.
\end{enumerate}
A globalization is said to be admissible if
\begin{enumerate}
\item[(GL5)]  $B=\sum_{g\in G} \beta_g (\varphi (A))$.
\end{enumerate}
\end{defi}

Once $\varphi (A)\trianglelefteq B$ then for every $g\in G$ the subset $\beta_g (A)$ is an ideal of $B$ as well. The main results about globalization of partial actions of groups on algebras can be resumed in the following Theorem.

\begin{thm} \label{globthmgroup} \cite{dok} Let $\alpha$ be a partial action of a group $G$ on an algebra $A$, then the following statements hold.
\begin{enumerate}
\item The partial action $\alpha$ admits an enveloping action if and only if the partial action is unital, that is, each ideal $A_g$ is generated by a central idempotent $1_g \in A$ and each $\alpha_g$ is a unital isomorphism between the algebras $A_{g^{-1}}$ and $A_g$. Moreover, this enveloping action is unique up to isomorphism.
\item If $\alpha$ admits an enveloping action $(B, \beta , \varphi )$, then the partial skew group algebra $A\rtimes_{\alpha} G$ is Morita equivalent to the skew group algebra $B\rtimes_{\beta} G$.
\end{enumerate}
\end{thm}

The Globalization Theorem \ref{globthmgroup} reveals itself as a useful tool in the algebraic theory of partial actions. A remarkable example comes from the theory of partial Galois extensions \cite{DFP}. Consider a unital algebra $S$ and let $\alpha$ be a unital partial action of a finite group $G$ on $S$. Define the trace map of the partial action $t^{par}_{S/R}:S\rightarrow R$, in which $R=S^{\alpha}$ is the subalgebra of partial invariants, as
\[
t^{par}_{S/R} (s)=\sum_{g\in G} \alpha_g (s1_{g^{-1}}).
\]
Let $(S',\beta, \varphi)$ be the enveloping of the unital partial action $\alpha$. As the group is finite, the algebra $S'$ is unital.  Let $R'=S'^{\beta}$ be the subalgebra of invariants of $S'$, and $t_{S'/R'} :S' \rightarrow R'$ be the classical trace given by
\[
t_{S'/R'} (x)=\sum_{g\in G} \beta_g (x) .
\]
It is possible to relate the partial trace and the global trace and, in particular, one can conclude that $t^{par}_{S/R}$ is onto $R$ if and  only if $t_{S'/R'}$ is onto $R'$. Moreover, one can prove that $S$ is a partial Galois extension over $R$ if and only if $S'$ is a Galois extension over $R'$. These results allows one to generalize the classical theorems of Galois extensions of commutative rings \cite{CHR} to the case of Galois extensions by partial actions.

Globalization theorems for partial actions of Hopf algebras first appeared in \cite{AB}. The very notion of an enveloping action, or globalization, needed to be slightly modified. In the case of groups, if $\beta :G\rightarrow \mbox{Aut}(B)$ is a global action and $I\trianglelefteq B$ is an ideal, the translated $\beta_g (I)$ is also an ideal of $B$ for every $g\in G$. For Hopf algebras this is not the case, if $B$ is an $H$-module algebra, $I$ is an ideal of $B$ and  $h\in H$ , then the subspace $h\triangleright I$ need not to be an ideal at all. Then, many tricks involving ideals in the globalization theorem for partial actions of groups were no longer possible in the Hopf case.

\begin{defi} \label{globhopf} \cite{AB} An enveloping action (globalization) of a symmetric partial action $\cdot :H\otimes A \rightarrow A$ of a Hopf algebra $H$ on a unital algebra $A$ is a triple $(B, \triangleright, \varphi )$, in which
\begin{enumerate}
\item[(GL1)] $B$ is a not necessarily unital algebra with an $H$-module structure given by $\triangleright :H\otimes B \rightarrow B$ satisfying $h\triangleright (ab)=(h_{(1)}\triangleright a)(h_{(2)}\triangleright b)$;
\item[(GL2)] $\varphi :A\rightarrow B$ is a multiplicative $k$-linear map and $\varphi (1_A)$ is a central idempotent in $B$, turning $\varphi (A)$ into a unital ideal in $B$;
\item[(GL3)] for each $a\in A$ and $h\in H$, we have $\varphi (h\cdot a) =\varphi (1_A )(h\triangleright \varphi (a))$.
\end{enumerate}
A globalization is said to be admissible if
\begin{enumerate}
\item[(GL4)]  $B=H\triangleright \varphi (A)$.
\end{enumerate}
A globalization is said to be minimal if
\begin{enumerate}
\item[(GL5)] for any left $H$-submodule $M\leq B$, the condition $\varphi (1_A )M=0$ implies that $M=0$.
\end{enumerate}
\end{defi}

The axiom (GL3) states that the map $\varphi$ is a morphism of left partial $H$-module algebras between $A$ and $\varphi (A)$ with the induced partial action coming from the global action $\triangleright :H\otimes B \rightarrow B$. The axiom (GL4) of an admissible globalization is not enough to garantee the minimality of the enveloping action and consequently the uniqueness of the globalization up to isomorphism. To obtain the uniqueness, one needs to impose also the axiom (GL5), which is automatically valid in the group case. Again we are dealing with symmetric partial actions, although there is a nonsymmetric version of the definition of globalization in which $\varphi (A)$ is only required to be a right ideal in $B$ and $\varphi (1_A )$ is only an idempotent, not necessarily central, in $B$.

The main results about enveloping actions of partial actions of Hopf algebras can be put into a single theorem:

\begin{thm} \label{globthmhopf} \cite{AB} Let $\cdot :H\otimes A \rightarrow A$ be a symmetric partial action of a Hopf algebra $H$ then the following statements hold.
\begin{enumerate}
\item There is, up to isomorphism, a unique minimal enveloping action $(B,\triangleright ,\varphi )$.
\item If $(C, \blacktriangleright ,\theta )$ is another admissible (not necessarily minimal) enveloping action of $A$, then there is a surjective morphism of $H$-module algebras $\Phi :C\rightarrow B$ such that $\Phi \circ \theta =\varphi$, where $B$ is as in (1). If $(C, \blacktriangleright ,\theta )$ is minimal, then $\Phi$ is an isomorphism.
\item The partial smash product $\underline{A\# H}$ is Morita equivalent to the smash product $B\# H$, where $(B,\triangleright ,\varphi )$ is the minimal enveloping action.
\end{enumerate}
\end{thm}

\begin{proof} The techniques were largely inspired by the proof of Theorem \ref{globthmgroup} of globalization for partial actions of groups. In this note, we only present a sketch of the construction of the enveloping action.

Consider the $k$-algebra $\mbox{Hom}_k (H,A)$ with the convolution product.  The formula
$(h\triangleright f)(k)=f(kh)$.
defines a structure of a left $H$-module algebra on $\mbox{Hom}_k (H,A)$. Now define the linear map 
$\varphi :\ A\rightarrow \mbox{Hom}_k (H,A)$,  $\varphi (a)(h)=h\cdot a$.
It is easy to prove that $\varphi$ is injective and multiplicative. Finally, Let $B$ be the subalgebra of $\mbox{Hom}_k (H,A)$ generated by elements of the form $h\triangleright \varphi (a)$, with $h\in H$ and $a\in A$. This subalgebra is by construction an $H$-submodule algebra and the fact that $\varphi (A)$ is an ideal of $B$ follows from the identities
\[
\varphi (a)(h\triangleright \varphi (b))=\varphi (a(h\cdot b)), \qquad (h\triangleright \varphi (b))\varphi (a)=\varphi ((h\cdot b)a) .
\]
In fact, these identities also prove that $\varphi (1_A)$ is a central idempotent in $B$. Moreover, the algebra $B$ coincides with the subspace $H\triangleright \varphi (A)$, this follows from the identity
\[
(h\triangleright \varphi (a))(k\triangleright \varphi (b))=h_{(1)}\triangleright \varphi (a(S(h_{(2)})k\cdot b)) .
\]
This gives the basic construction of the so called ``standard enveloping action'', the verification of the axioms (GL1-GL5) of Definition \ref{globhopf} are straightforward. The proof of items (2) and (3) can be found in \cite{AB}.
\end{proof}

The Globalization Theorem can be applied in order to prove an analog for partial smash products of he Blattner-Montgomery Theorem \cite{BM},
see \cite{AB3}. This generalizes a result of Lomp in \cite{L}, which consists of a generalization of the Cohen-Montgomery Theorem to partial skew group algebras. In the classical case,  we have an isomorphism,
\[
(A\# H)\# H^\circ \cong A\otimes \mathcal{L} , 
\]
in which $H$ is a residually finite dimensional Hopf algebra, $H^\circ$ is its finite dual Hopf algebra, which is dense in $H^*$, and $\mathcal{L}$ is a dense subring of $\mbox{End}_k (H)$. If $H$ is finite dimensional, we can rewrite this isomorphism as
\[
(A\# H)\# H^* \cong A\otimes \mbox{End}_k (H).
\]
Consider a finite dimensional Hopf algebra $H$ and a left partial $H$-module algebra $A$. We have a map
between the double smash product $(\underline{A\# H})\# H^*$ and the tensor product $A\otimes \mbox{End}_k (H)$,
which is in general not an isomorphism. This map can be obtained using a (not necessarily minimal) enveloping action $(B,\triangleright , \theta)$ such that $B$ is a unital algebra. There is a natural inclusion  
\[
\imath : (\underline{A\# H})\# H^* \rightarrow (B\# H)\# H^* , 
\]
and there is the classical isomorphism 
\[ 
\Phi : (B\# H)\# H^* \rightarrow  B\otimes \mbox{End}_k (H) .
\]
$\Phi \circ \imath$ corestricts to a map
\[
\Phi \circ \imath :(\underline{A\# H})\# H^* \rightarrow A\otimes \mbox{End}_k (H).
\]
which is the restriction of the map $\Phi$ above whose image lies in $A\otimes \mbox{End}_k (H)$. The kernel and image of
this map can be described.

\section{Partial representations and partial modules}\label{5}

As we have seen before, in order to define a partial action of a group $G$ on a set $X$, we have to assign to each element $g\in G$ a bijection $\alpha_g :X_{g^{-1}}\rightarrow X_g$ between two subsets of $X$ such that the compositions of these bijections, wherever they are defined, should be compatible with the group operation. Given the set $X$, one can define the set ${\mathcal I}(X)$ of partially defined bijections on $X$:
\[
\mathcal{I} (X)=\{ f:\mbox{Dom}(f)\subseteq X \rightarrow \mbox{Im}(f)\subseteq X \; | \; f \mbox{  is a bijection } \}.
\]
The composition between two elements $f_1, f_2 \in \mathcal{I} (X)$ is a new element $f_2 \circ f_1 \in \mathcal{I}(X)$ whose domain is 
$\mbox{Dom}(f_2 \circ f_1 )=f_1^{-1}(\mbox{Im}(f_1 )\cap \mbox{Dom}(f_2 ))$ and image $\mbox{Im}(f_2 \circ f_1 )=f_2 (\mbox{Im}(f_1 )\cap \mbox{Dom}(f_2 ))$. Note that if $\mbox{Im}(f_1) \cap \mbox{Dom}(f_2)=\emptyset$, then the composition $f_2\circ f_1$ is defined to be the empty map  $\emptyset :\emptyset \subseteq X \rightarrow \emptyset \subseteq X$. This composition operation endows $\mathcal{I}(X)$ with the structure of an inverse semigroup. 

\begin{defi} \cite{Law} An inverse semigroup is a set $S$ with an associative operation $\cdot :S\times S \rightarrow S$ such that, for each $s\in S$ there exists a unique element $s^* \in S$ satisfying $s\cdot s^* \cdot s =s$ and $s^* \cdot s \cdot s^* =s^*$, this element $s^*$ is called the pseudo-inverse of $s$.
\end{defi}

In fact, $\mathcal{I}(X)$ is an inverse monoid, because the element $\mbox{Id}_X \in \mathcal{I} (X)$ is the neutral element with respect to the composition, moreover, this inverse monoid also has a zero element which is the empty map previously defined. In fact, Wagner-Preston theorem \cite{P,W} states that every inverse semigroup is a semigroup of the form $\mathcal{I}(X)$ for some set $X$. 

Partial group actions can be rephrased in terms of the theory of inverse semigroups \cite{Law}, which provide a more general and more suitable language to treat partial actions. A partial action of a group $G$ on a set $X$ can be viewed as a map $\alpha :G\rightarrow \mathcal{I}(X)$ such that $\mbox{Dom}(\alpha (g))=X_{g^{-1}}$ and $\mbox{Im}(\alpha (g))=X_g$. One can easily verify that the map $\alpha$ satisfy the following identities:
\begin{eqnarray}
\alpha (e) & = & \mbox{Id}_X,  \label{pr1}\\
\alpha (g) \alpha (h) \alpha (h^{-1}) & = & \alpha (gh)\alpha (h^{-1}),\qquad \forall g,h \in G , \label{pr2} \\
\alpha (g^{-1}) \alpha (g) \alpha (h) & = & \alpha (g^{-1}) \alpha (gh),  \qquad \forall g,h \in G , \label{pr3}
\end{eqnarray}
the map $\alpha$ is called a partial representation of $G$. In general, we can define a partial representation of a group $G$ on an inverse semigroup $S$ as a map $\pi :G\rightarrow S$ satisfying the identities (\ref{pr1}), (\ref{pr2}) and (\ref{pr3}). 

Since every group $G$ is in particular an inverse semigroup, one could expect that a partial action of $G$ on a set $X$ (or a partial representation of $G$ on a semigroup $S$) is a morphism between inverse semigroups from $G$ to ${\mathcal I}(X)$ (resp.\ from $G$ to $S$). However, this is not the case and what we find instead is a composition law which works only in the presence of a ``witness'' on the left or on the right. In order to recover a morphism of inverse semigroups associated to a partial representation, there is a universal semigroup constructed upon the group $G$, called Exel's semigroup \cite{ruy2}:
\[
S(G)=\langle [g] \; | \; g\in G, \; \; [e]=1, \; \; [g][h][h^{-1}]=[gh][h^{-1}], \; \; [g^{-1}][g][h]=[g^{-1}][gh] \rangle .
\]
Using the defining relations of $S(G)$ we conclude that 
$[g][g^{-1}][g]=[gg^{-1}][g]=[e][g]=[g]$ proving that
$S(G)$ is indeed an inverse semigroup. The map $[\; ]:G\rightarrow S(G), g\mapsto [g]$ is then a partial representation of $G$ on $S(G)$ and $S(G)$ satisfies the following universal property: given an inverse semigroup $S$ and a partial representation $\pi :G\rightarrow S$, there exists a unique morphisms of semigroups $\widehat{\pi}:S(G)\rightarrow S$ such that $\pi =\widehat{\pi}\circ [\; ]$. 

Kellendonk and Lawson \cite{KL} proved that Exel's inverse semigroup can be characterized as the Birget-Rhodes expansion of the group $G$. Given a group $G$, the Birget-Rhodes expansion \cite{BR1,BR2,Sz} is the inverse semigroup with underlying set given by 
\[
\widetilde{G}^R =\{ (A,g)\in \mathcal{P}(G)\times G \; | \; e\in A , \;\ \; g\in A \} , 
\]
and with the operation
\[
(A,g)(B,h)=(A\cup gB , gh) .
\]
The pseudo-inverse of an element $(A, g)\in \widetilde{G}^R$ is the element $(g^{-1}A, g^{-1})$. 
The isomorphism between $S(G)$ and $\widetilde{G}^R$ is performed by the map $\phi :S(G)\rightarrow \widetilde{G}^R$, defined as 
$\phi ([g])=(\{ e,g \}, g)$. This morphism is created out of a partial representation of $G$ on $\widetilde{G}^R$, using the universal property of $S(G)$. It is easy to see that $\phi$ in injective, and the surjectivity follows from the standard presentation of an element of $S(G)$. We shall return to this point later, when we discuss the linearized version of this semigroup, the partial group algebra.

Let us move further to the linearized case to consider partial representations of a group $G$ on algebras and partial $G$-modules.

\begin{defi} \cite{dok} A partial representation of a group $G$ in an algebra $B$ is a map $\pi :G\rightarrow B$   such that 
\begin{enumerate}
\item $\pi (e) = 1_B$,
\item $\pi (g)\pi (h) \pi (h^{-1}) = \pi (gh)\pi (h^{-1}),$
\item $\pi (g^{-1}) \pi (g) \pi (h) = \pi (g^{-1}) \pi (gh),$
\end{enumerate} 
for all $g, h\in G$.
Let $\pi:G\rightarrow B$ and $\pi':G\rightarrow B'$ be two partial representations of $G$. A morphism of partial partial representations is an
algebra morphism $f:B\rightarrow B'$ such that $\pi ' =f\circ \pi$.
The category of partial representations of $G$, denoted as $\underline{ParRep}_G$, is the category whose objects are pairs $(B,\pi)$, where $B$ is a unital $k$-algebra and $\pi : G\rightarrow B$ is a partial representation of $G$ in $B$, and whose morphisms are morphisms of partial representations.
\end{defi}

It is evident that representations of $G$ are automatically partial representations. A partial representation $\pi:G\to B$ satsfying the condition
$\pi (g) \pi (g^{-1}) =1_B$, for all $g\in G$,  turns out to be a usual representation of the group $G$ on the algebra $B$. As classical representations of a group $G$ are directly related with $G$-modules, one can use the concept of a partial representation of $G$ to define partial $G$-modules. A partial $G$-module is a pair $(X,\pi)$ where $X$ is a $k$-module and $\pi:G\to \End_k(X)$ is a partial representation of $G$. A morphism of partial $G$ modules $f:(X,\pi)\rightarrow (Y,\pi')$ is a $k$-linear map  $f:X\to Y$  satisfying $\pi'(g)\circ f=f\circ \pi(g)$ for all $g\in G$. We will denote by $\underline{ParMod}_G$ the category of partial $G$-modules.

\begin{exmp} Given a unital partial action $(\{ A_g \}_{g\in G} , \{ \alpha_g \}_{g\in G} )$ of $G$ on a unital algebra $A$, we can define partial representations
$$\pi :\ G \to \mbox{End}_k (A),~~\pi (g)(a) =\alpha_g (a1_{g^{-1}}),$$
and
$$\psi:\ G\to A\rtimes_{\alpha} G,~~\psi(g)=1_g \delta_g.$$
\end{exmp}

\begin{exmp} A partial $G$-module $(X,\pi )$ induces a partial action of $G$ on the $k$-module $X$. Since $\pi$ is a partial representation, the operators $P_g =\pi (g) \pi (g^{-1})$ are projections. Take the domains $X_g =P_g (X)$ and the partially defined morphisms $\alpha_g =\pi (g)|_{{}_{X_{g^{-1}}}} :\ X_{g^{-1}}\rightarrow X_g $. These data define a partial action of $G$ on $X$.

Conversely, consider a partial action of $G$ on a $k$-module $X$, and assume that there exists a projection $P_g :X\rightarrow X_g$ such that $P_g (X)=X_g$, and that the maps $\alpha_g:X_{g^{-1}}\to X_g$ are $k$-linear. Then we have a partial representation $\pi :G\rightarrow \mbox{End}_k (X)$, $\pi (g) (x)=\alpha_g (P_{g^{-1}}(x))$.
\end{exmp}

A well-nown result in classical representation theory states that representations $\rho$ of a group $G$ on a $k$-module $M$ correspond
to algebra morphisms from the group algebra $kG$ to the algebra $\mbox{End}_k (M)$. In order to obtain a similar result for
partial representations, we define the partial group algebra $k_{par}G$ as the unital algebra generated by symbols $[g]$, with $g\in G$, satisfying the relations
\[
[e]=1 \; , \quad [g][h][h^{-1}]=[gh][h^{-1}]\; , \quad [g^{-1}][g][h]=[g^{-1}][gh],\]
for all $ g,h\in G$.
It is easy to see that the map $[ \underline{\, }] :\ G\to k_{par}G$ is a partial representation of $G$ on the algebra $k_{par}G$. This algebra can be easily recognized as the algebra $kS(G)$, which is the linearized version of Exel's inverse semigroup $S(G)$. The partial group algebra factorizes partial representations of the group $G$ by algebra morphisms as stated in the following proposition.

\begin{prop} \cite{dok0} For each partial representation $\pi :G \rightarrow B$ there exists a unique algebra morphism $\hat{\pi} :k_{par}G\rightarrow B$, such that $\pi =\hat{\pi}\circ [\underline{\, }]$. Conversely, given any algebra morphism 
$\phi :k_{par}G\rightarrow B$, one can define a partial representation $\pi_{\phi} :G\rightarrow B$ such that $\phi =\hat{\pi}_{\phi}$. 
\end{prop}

The most remarkable properties of the partial group algebra $k_{par}G$ come from its rich internal structure. First, let us highlight an important commutative subalgebra of the partial group algebra. For each $g\in G$ define the elements $\varepsilon_g =[g][g^{-1}]\in k_{par}G$. 
For all $g,h \in G$, we have the following identities
$$[g]\varepsilon_h= \varepsilon_{gh}[g]~~;~~\varepsilon_g  = \varepsilon_g \varepsilon_g~~;~~\varepsilon_g \varepsilon_h = \varepsilon_h \varepsilon_g.$$
Let $A$ be the subalgebra of $Ak_{par}G$ generated by all elements of the form $\varepsilon_g$, with $g\in G$. $A$ is commutative and it is generated by central idempotents. There is a very natural unital partial action of $G$ on $A$ given by the domains $A_g =\varepsilon_g A$ and the partially defined isomorphisms $\alpha_g :A_{g^{-1}}\rightarrow A_g$ defined as $\alpha_g (a)=[g]a[g^{-1}]$, in which $a=\varepsilon_{g^{-1}}a' \in A_{g^{-1}}$, more explicitly
\[
\alpha_g (\varepsilon_{g^{-1}}\varepsilon_{h_1}\ldots \varepsilon_{h_n})=\varepsilon_{gh_1}\ldots \varepsilon_{gh_n} \varepsilon_g .
\]

\begin{thm} \cite{dok} $G$ acts partially on the commutative algebra $A\subseteq k_{par}G$, and$k_{par}G \cong A\rtimes_{\alpha}G$.
\end{thm}

One of the consequences of this theorem is that a typical element $x\in k_{par}G$ can be written as a linear combination of monomials of the form $\varepsilon_{h_1}\ldots \varepsilon_{h_n}[g]$.

For a finite group $G$, the partial group algebra $k_{par}G$ has another characterization as a groupoid algebra. Consider the set
\[
\Gamma (G) =\{ (g,A)\in \mathcal{P}_0(G) \times G \, | \, g^{-1}\in A \} ,
\]
where $\mathcal{P}_0(G)$ denotes the subset of $\mathcal{P}(G)$ of the subsets of $G$ containing the neutral element $e\in G$. This set can be endowed with a structure of a groupoid with the operation in $\Gamma(G)$ defined as
$$(g,A)(h,B) =\begin{cases}
(gh, B)&{\rm if~}A=hB \\
\underline{\qquad} & \mbox{ otherwise.}
\end{cases}$$
The source and target map are given by the formulas $s(g,A)=(e,A)$ and $t(g,A)=(e,gA)$, and the inverse of $(g,A)$ is
$(g,A)^{-1}=(g^{-1}, gA)$. The partial group algebra $k_{par}G$ and the groupoid algebra $k\Gamma (G)$ are isomorphic \cite{dok0}. The isomorphism is constructed upon a partial representation of $G$, $\lambda_p :G\rightarrow k\Gamma (G)$, given by
\[
\lambda_p (g) =\sum_{A\ni g^{-1}} (g,A), 
\]
and using the universal property of $k_{par}G$. If $|G|=n$ then the dimension of $k_{par}G$ is $d=2^{n-2}(n+1)$, and moreover, this algebra is a direct sum of matrix algebras, see Theorem \ref{matrix}. Indeed, the connected components of the graph of the groupoid $\Gamma (G)$ reveils the richness of the internal structure of the algebra $k_{par}G\cong k\Gamma (G)$.

\begin{thm} \label{matrix}
\cite{dok0} The groupoid algebra $k\Gamma (G)$ has the direct sum decomposition
\[
k\Gamma (G) =\bigoplus_{\begin{array}{c}H\leq G\\ 1\leq m \leq (G:H) \end{array}} c_m (H) M_m (kH) .
\]
running over the subgroups $H\leq G$ and integers $m$ between $1$ and $(G:H)$. The multiplicities are given by
the recursive formula
$$c_m (H)=\frac{1}{m} \left({(G:H)-1 \choose m-1}-m \sum_{\begin{array}{c} H<B\leq G  \\  (B:H)|m \end{array}} \frac{c_{\frac{m}{(B:H)}} (B)}{(B:H)}\right).$$
\end{thm}

In particular, if $k$ is a field and ${\rm char}(k)\nmid  |G|$, then $k\Gamma (G)$ is semisimple, and the matrix decomposition in
the Theorem  is a Wedderburn-Artin decomposition.

A theory of partial representations of Hopf algebra has been developed in \cite{ABV}. 

\begin{defi} \cite{ABV} \label{partialrep}
Let $H$ be a Hopf $k$-algebra, and let $B$ be a unital $k$-algebra. A \emph{partial representation} of $H$ in $B$ is a linear map $\pi: H \rightarrow B$ such that 
\begin{enumerate}[{(PR1)}]
\item $\pi (1_H)  =  1_B$, \label{partialrep1}
\item $\pi (h) \pi (k_{(1)}) \pi (S(k_{(2)}))  =   \pi (hk_{(1)}) \pi (S(k_{(2)})) $,
\label{partialrep2}
\item $\pi (h_{(1)}) \pi (S(h_{(2)})) \pi (k)  =   \pi (h_{(1)}) \pi (S(h_{(2)})k)$, \label{partialrep3} 
\item $\pi (h) \pi (S(k_{(1)})) \pi (k_{(2)}) = \pi (hS(k_{(1)})) \pi (k_{(2)})$, \label{partialrep4}
\item $\pi (S(h_{(1)}))\pi (h_{(2)}) \pi (k) = \pi (S(h_{(1)}))\pi (h_{(2)} k)$. \label{partialrep5}
\end{enumerate}
A morphism between two partial representations $(B,\pi)$ and $(B',\pi')$ of $H$ is
an algebra morphism $f:B\rightarrow B'$ such that $\pi ' =f\circ \pi$. $\underline{ParRep}_H$
is the category of partial representations.
\end{defi}

Evidently, every representation of $H$ is a partial representation. Conversely, every partial representation $\pi:H\to B$ satisfying one of the conditions,
\[
\pi (h_{(1)})\pi (S(h_{(2)}))=\epsilon (h) 1_B, \qquad \pi (S(h_{(1)})) \pi (h_{(1)}) =\epsilon (h) 1_B,
\]
for all $h\in H$,
is a representation of $H$. If the antipode is invertible, then 
(PR3) and (PR4) can be rewritten as
\begin{enumerate}
\item[(PR3)] $\pi (\anti (h_{(2)})) \pi (h_{(1)}) \pi (k)  =   \pi (\anti (h_{(2)})) \pi (h_{(1)}k)$,
\item[(PR4)] $\pi (h) \pi (k_{(2)}) \pi (\anti (k_{(1)})) = \pi (hk_{(2)}) \pi (\anti (k_{(1)}))$.
\end{enumerate}
If $H$ is commutative or cocommutative, then $S^{-1}=S$ and (PR3) and (PR4) folllow from (PR1), (PR2) and (PR5). 
In particular, if $H=kG$ is a group algebra, partial representations of $kG$ coincide with partial representaions of the group $G$, as defined before. As in the case of groups, we consider a variant on the category of partial representations, the category of partial $H$-modules
${}_H \mathcal{M}^{par}$. Objects in this category are pairs $(X,\pi)$, where $X$ is a $k$-module and $\pi:H\to \End_k(X)$ is a partial representation. A morphism $f:(X,\pi)\to (Y,\pi')$ is a $k$-linear map $f:X\to Y$ satisfying $f\circ \pi(h)=\pi'(h)\circ f$ for all $h\in H$. 

As paradigmatic examples of partial representations of a Hopf algebra, we have:

\begin{exmp} \cite{ABV}
Let $\cdot :H\otimes A \rightarrow A$ be a partial action of a Hopf algebra $H$ on a unital $k$-algebra $A$. We have partial representations $\pi :\ H\rightarrow \mbox{End}_k (A)$, $\pi (h) (a)=h\cdot a$ and
$\psi:\ H\rightarrow \underline{A\# H}$, $\psi (h)=1_A \# h =(h_{(1)}\cdot 1_A )\otimes h_{(2)} $.
\end{exmp}

As in the group case, there is a universal algebra factorizing partial representations by algebra morphisms.

\begin{defi}\label{definitionhpar} \cite{ABV} Let $H$ be a Hopf algebra 
and let $T(H)$ be the tensor algebra of the vector space $H$. The {\em partial Hopf algebra} $H_{par}$ is the quotient
of $T(H)$ by the ideal $I$ generated by elements of the form 
\begin{enumerate}
\item $1_H - 1_{T(H)}$; 
\item $h \otimes k_{(1)} \otimes S(k_{(2)}) - hk_{(1)} \otimes S(k_{(2)})$;
\item $h_{(1)} \otimes S(h_{(2)}) \otimes k - h_{(1)} \otimes S(h_{(2)})k$;
\item $h \otimes S(k_{(1)}) \otimes k_{(2)} - hS(k_{(1)}) \otimes k_{(2)}$;
\item $S(h_{(1)}) \otimes h_{(2)} \otimes k - S(h_{(1)}) \otimes h_{(2)}k$.
\end{enumerate}
\end{defi}

Let $[h]$ be the class in $H_{par}$ represented by $h\in H$ in $H_{par}$ by $[h]$. The map 
$[\underline{\; }]:\ H \to H_{par}$ has the following properits
\begin{enumerate}
\item $[\alpha h + \beta k] = \alpha [h]+ \beta [k]$;
\item $[1_H] = 1_{H_{par}}$; 
\item $[h][k_{(1)}][S(k_{(2)})] = [hk_{(1)}][S(k_{(2)})]$;
\item $[h_{(1)}][S(h_{(2)})][k] = [h_{(1)}][S(h_{(2)})k]$;
\item $[h][S(k_{(1)})][k_{(2)}] = [hS(k_{(1)})][k_{(2)}]$;
\item $[S(h_{(1)})][h_{(2)}][k] = [S(h_{(1)})][h_{(2)}k]$,
\end{enumerate}
for all $\alpha, \beta \in k$ and $h,k \in H$. $[\underline{\;}]$ is a partial representation of the Hopf algebra $H$ on $H_{par}$.
The partial Hopf algebra $H_{par}$ has the following universal property.

\begin{thm}\label{univHpar} \cite{ABV}
For every partial representation $\pi: H \rightarrow B$ there is a unique morphism of algebras $\hat{\pi}: H_{par}
\rightarrow B$ such that 
$\pi = \hat{\pi} \circ  [ \underline{\; }  ]$. Conversely, given an algebra morphism $\phi : H_{par} \rightarrow B$,
there exists a unique partial representation $\pi_{\phi} :H\rightarrow B$ such that $\phi =\hat \pi_{\phi}$.
\end{thm}

The identity $\mbox{Id}_H :H\rightarrow H$ is a partial representation of $H$ on $H$ and factorizes through $H_{par}$, generating an algebra morphism $\partial :H_{par}\rightarrow H$, given by $\partial ([h^1]\ldots [h^n])=h^1 \ldots h^n$. In particular $H$ is a direct summand of $H_{par}$ as a $k$-module. 

An immediate consequence of Theorem \ref{univHpar} is that any partial $H$-module $(X,\pi)$, is an $H_{par}$-module. Indeed, the partial representation $\pi :H\rightarrow \mbox{End}_k (X)$ factorizes through a morphism of algebras $\hat{\pi} :H_{par} \rightarrow \mbox{End}_k (X)$, which makes $X$ into an $H_{par}$-module. Therefore, the category ${}_H \mathcal{M}^{par}$ of partial $H$-modules is isomorphic to the category ${}_{H_{par}}\mathcal{M}$ of $H_{par}$-modules.

\begin{exmp} For $H=kG$, the universal algebra $H_{par}$ is the partial group algebra $k_{par}G$, introduced above.
\end{exmp}

\begin{exmp} Let $H=\mathcal{U}(\mathfrak{g})$ be the universal enveloping algebra of a Lie algebra $\mathfrak{g}$.
Then every partial representation of $H$ is a morphism of algebras. Indeed, for a
partial representation $\pi :\mathcal{U}(\mathfrak{g})\rightarrow B$, we can prove that $\pi (h) \pi (k) =\pi (hk)$,
for all 
$h,k\in \mathcal{U}(\mathfrak{g})$. This can be done by induction on the length of the monomials, $k=X_1 \ldots X_n \in \mathcal{U}(\mathfrak{g})$, with $X_i \in \mathfrak{g}$, for $i=1,\ldots ,n$ \cite{ABV}. 
\end{exmp}

\begin{exmp} \label{kc2} \cite{ABV} Let $k$ be a field with characteristic different from 2.
Then $H=(k C_2)^*\cong kC_2$. Let $\{p_e,p_g\}$ be the basis of $H$ consisting of the projections onto $ke$ and $kg$.
Then $1=p_e + p_g$ and $p_e p_g=0$. 
Let $x = [p_e]$, $y = [p_g]$ be the corresponding generators of $H_{par}$. The first defining equation for $H_{par}$
yields  $y = 1-x$, and, in particular, we conclude that $xy=yx$ in $H_{par}$. 

Since $H$  is cocommutative the next five families of equations that define $H_{par}$ are reduced to the first three of
them. Writing the equations explicitly with respect to the basis elements one obtains eight equations;  using the fact that $y = 1-x$, all these  equations are reduced to 
\[
x(2x-1)(x-1)  = 0, 
\]
which implies that  $H_{par}$ is isomorphic to the $3$-dimensional algebra 
\[
k[x]/\langle x(2x-1)(x-1) \rangle .
\] 
\end{exmp}

The ``partial Hopf algebra'' $H_{par}$ has also a structure of a partial smash product. For each $h \in H$, define the elements  
\[
\varepsilon_h = [h_{(1)}][S(h_{(2)})] , \quad \tilde{\varepsilon}_h =[S(h_{(1)})][h_{(2)}] \quad \in H_{par} .
\]
We have the following identities in $H_{par}$:
\begin{enumerate}
\item[(a)] $\varepsilon_k [h] = [h_{(2)}]\varepsilon_{S^{-1}(h_{(1)})k}$;
\item[(b)] $[h]\varepsilon_k = \varepsilon_{h_{(1)}k}[h_{(2)}]$;
\item[(c)] $\varepsilon_{h_{(1)}} \varepsilon_{h_{(2)}}= \varepsilon_h$;
\item[(d)] $\tilde{\varepsilon}_k [h] = [h_{(1)}] \tilde{\varepsilon}_{kh_{(2)}}$;
\item[(e)] $[h]\tilde{\varepsilon}_k = \tilde{\varepsilon}_{kS^{-1} (h_{(2)})} [h_{(1)}]$;
\item[(f)] $\tilde{\varepsilon}_{h_{(1)}} \tilde{\varepsilon}_{h_{(2)}} =\tilde{\varepsilon}_h$;
\item[(g)] $\tilde{\varepsilon}_h \varepsilon_k = \varepsilon_k \tilde{\varepsilon}_h$. 
\end{enumerate}
Unlike the group case, we don't construct a commutative subalgebra $A$ of $H_{par}$ but we define two subalgebras
\[
A=\langle \varepsilon_h \; | \; h\in H \rangle , \qquad \mbox{and  } \tilde{A}=\langle \tilde{\varepsilon}_h \; | \; h\in H \rangle ,
\]
which are not necessarily commutative, nevertheless, they commute mutually. With this subalgebra, one can prove the following Structure Theorem about $H_{par}$.

\begin{thm}\label{Hparisosmash}
Let $H$ be a Hopf algebra. There exists a partial action of $H$ on the subalgebra $A\subseteq H_{par}$ such that
$H_{par} \cong \underline{A\# H}$. 
\end{thm}

\begin{proof} (Sketch) Basically, the partial action of $H$ on $A$ is given by
\[
h\cdot (\varepsilon_{k^1} \ldots \varepsilon_{k^n })=[h_{(1)}]\varepsilon_{k^1} \ldots \varepsilon_{k^n }[S(h_{(2)})] =
\varepsilon_{h_{(1)} k^1} \ldots \varepsilon_{h_{(n)}k^n }\varepsilon_{h_{(n+1)}} .
\]
The morphisms $\Phi :H_{par} \rightarrow \underline{A\# H}$ and $\Psi :\underline{A\# H}  \rightarrow H_{par}$ are obtained using the universal property of $H_{par}$ and the universal property of the partial smash product. They are given by the formulas
\begin{eqnarray*}
\Phi ([h^1]\ldots [h^n])&=&\left( \varepsilon_{h^1_{(1)}} \varepsilon_{h^1_{(2)}h^2_{(1)}}\ldots \varepsilon_{h^1_{(n)}\ldots h^n_{(1)}} \right) \# h^1_{(n+1)}\ldots h^n_{(2)};\\
\Psi ((\varepsilon_{h^1} \ldots \varepsilon_{h^n} \# k)&=&\varepsilon_{h^1} \ldots \varepsilon_{h^n}[k] .
\end{eqnarray*}
\end{proof}

With the help of this Structure Theorem, we can construct a surprising example of a finite dimensional Hopf algebra whose ``partial Hopf algebra'' is infinite dimensional, namely, the four dimensional Sweedler Hopf algebra. Let us denote
\[
H_4 =\langle 1,g,x \; | \; g^2 =1, \; x^2 =0, \; gx=-xg \rangle
\]
which is a Hopf algebra with structure given by
\begin{eqnarray*}
& \, & \Delta (g)=g\otimes g \qquad \epsilon (g) =1 , \qquad S(g)=g ;\\
& \, & \Delta (x)=x\otimes 1 +g\otimes x, \qquad \epsilon (x)=0 , \qquad S(x)=-gx .
\end{eqnarray*}
The subalgebra $A\subseteq (H_4)_{par}$ is the infinite dimensional algebra
\[
A=k[x,z]/\langle 2xz-z ,2x^2 -x \rangle 
\]
And then, the algebra $(H_4)_{par}$ can be written as the partial smash product \[
(H_4)_{par} =\underline{\left( k[x,z]/\langle 2xz-z ,2x^2 -x \rangle \right) \# H_4} .
\]
 Roughly speaking, the potential for partiality of a Hopf algebra is encoded in the size of the subalgebra $A\subseteq H_{par}$. 
As we see from the above example for the Sweedler Hopf algebra, even when $H$ is finite dimensional, $H_{par}$ can be infinite dimensional. 
This enabled the construction of an infinity of non-isomorphic finite dimensional partial modules over the Sweedler Hopf algebra \cite{ABV2}. 
 
One of the important features of the ``partial Hopf algebra'' $H_{par}$ is that it has a structure of a Hopf algebroid over the base algebra $A$. This enables us to explore the monoidal structure of the category of partial $H$-modules, which is isomorphic to the category of $H_{par}$-modules. The main results relative to the Hopf algebroid structure of $H_{par}$ and the monoidal strtucture of the category ${}_H \mathcal{M}^{par}$ are summarized in Theorem \ref{5.16}.

\begin{thm} \label{5.16}\cite{ABV} 
\begin{enumerate}
\item Let $H$ be a Hopf algebra with invertible antipode, then $H_{par}$ is a Hopf algebroid over the base algebra $A$.
\item There is a functor $(\underline{\quad})_{par}$, between the category of Hopf algebras with invertible antipode and 
the category of Hopf algebroids, sending $H$ to $H_{par}$.
\item The category of partial $H$-modules (or $H_{par}$-modules), is a monoidal category $({}_{H_{par}} \mathcal{M}, \otimes_A ,A)$ and the forgetful functor $U:{}_{H_{par}} \mathcal{M} \rightarrow {}_A \mathcal{M}_A$ is strictly monoidal.
\item Algebras in the monoidal category ${}_{H_{par}} \mathcal{M}$ coincide with partial $H$-modules algebras.
\end{enumerate}
\end{thm}

\begin{exmp} Consider a field of characteristic different from 2 and the Hopf algebra $H=(kC_2)^*$. As we have seen in Example \ref{kc2}, the Hopf algebroid $H_{par}$ is the three dimensional algebra $k[x]/\langle x(x-1)(2x-1)\rangle$. Then the partial $H$-modules, which can be interpreted as partially $\mathbb{Z}_2$-graded vector spaces are of the form
\[
X=X_0 \oplus X_1 \oplus X_{\frac{1}{2}} .
\]
The algebra objects in the category of partial $H$-modules are the partially $\mathbb{Z}_2$-graded algebras which are described in Theorem \ref{z2grad}. 
\end{exmp}

The theory of Hopf algebroids may provide a fresh and wider environment to treat partial representations and partial modules. The monoidal structure of the category of partial modules gives rise to many new questions form the point of view of representation theory. 
 
\section{Conclusions and perspectives}
\begin{flushleft}
{\it For we know in part... \\ 
but when that which is perfect is come,\\
then that which is in part shall be done away.\\
 I Corinthians 13:9-10}\\
\end{flushleft}

So far, the history of partial actions spans two decades since their first appearance in the context of operator algebras.
We conclude this survey opening up some possible new research lines for the years to come.\\

(1) The existing theory mostly deals with discrete groups.
 In order to take into account the topological or differentiable structure of the group some extra assumptions have to be made in order to make the domains vary coherently with the elements of the group. The simplest case of partial actions coming from incomplete flows of vector fields on a differentiable manifold \cite{AbadieTwo} remains to be fully understood. 
As we mentioned before, the globalization of these partial actions give rise to non-Hausdorff manifolds, and the techniques of noncommutative differential geometry may be needed to treat these globalized spaces. A concrete problem with physical interest is the resolution of cosmic singularities. The famous Hawking-Penrose theorem states that very generic conditions of causality in the space time and with positive energy conditions for gravity will force the space time to have incomplete causal geodesics \cite{Haw}. Then the globalization of this partial action of the additive group of real numbers, which would be the candidate of ``space time'' with a complete set of causal geodesiscs, would be a non-Hausdorrf space, perhaps better described in terms of noncommutative geometry. 

(2) Partial representations of finite groups have been investigated intensively \cite{dok0,dok1}. The machinery to find irreducible partial representations of a finite group $G$ uses explicitly the structure of the groupoid $\Gamma (G)$ described in Section \ref{5}. For an infinite discrete group, it is also possible to characterize the partial group algebra $k_{par}G$ as an \'etale  groupoid of germs as described in \cite{BEM}. But for a general topological or Lie group, the description of the partial group algebra as a groupoid algebra has not been achieved. Even the classification of finite dimensional partial representations of a compact Lie group needs to be done. An interesting object to be analysed is the Hopf algebroid of representative functions of $k_{par}G $, for $G$ a compact Lie group. Furthermore, the theory of Lie groupoids \cite{Mack} could provide a useful language to describe partial actions of Lie groups over smooth manifolds.

(3) One can construct a partial representation of a given Hopf algebra $H$ out of a classical $H$-module with a projection on it satisfying certain commutation relations \cite{ABV2}. The notion of a globalization can be extended to partial representations. This is what we call a dilation of a partial representation and it leads to a functor from the category of partial $H$-modules to the category of $H$-modules with projections \cite{ABV2}. The relationship between the subcategory of $H$-modules inside the category of partial $H$-modules need to be better clarified. The role of the monoidal structure over $A\subseteq H_{par}$ in the category of partial 
$H$-modules and the monoidal structure over the base field $k$ in the category of $H$-modules remains elusive. 

(4) The theory of partial corepresentations and partial comodules of Hopf algebras has been developed in \cite{ABCQV}. For a Hopf algebra $H$, a universal coalgebra $H^{par}$is constructing, factorizing any partial corepresentation by a morphism of coalgebras in a way dual to how the partial ``Hopf algebra'' $H_{par}$ factorizes partial representations by morphisms of algebras. This universal coalgebra has moreover a Hopf coalgebroid structure that might shed some light on the monoidal structure of the category of partial $H$-comodules and it allows one to construct new categories, as the category of partial Yetter-Drinfeld modules.

(5) The categorification of partial actions is another direction to be explored. The first approach appeared in \cite{AAB}, where a partial action of a Hopf algebra on a $k$-linear category was considered. More recently, a new approach, of Hopf categories, emerged in \cite{BCV}. Basically, a Hopf category is a multi object analog of a Hopf algebra. In particular, groupoids provide examples of Hopf categories enabling us to describe partial actions of groups on sets by Hopf categories. Hopf categories may become a useful technique for the study of partial actions.

(6) There is much to be done to fully understand partial actions from a purely categorical point of view. When we study the partial representations of a Hopf algebra $H$, we end up with the Hopf algebroid $H_{par}$. Both $H$ and $H_{par}$ induce a Hopf monad on an appropriate monoidal category. The transition between partial $H$-modules and $H$-modules suggests that there is a close relation between these two different Hopf monads. Such a monadic approach to partial actions could lead to a better understanding of the role of partiality in representation theory.

\end{document}